\documentclass[preprint,12pt]{elsarticle}
\usepackage{amssymb,amsthm}
\usepackage{mathrsfs}
\usepackage{amsfonts}
\usepackage{mathtools}
\usepackage{bm}
\usepackage[dvipsnames]{xcolor}
\newcommand{\red}[1]{{\color{red} #1}}

\newcommand{\blue}[1]{{\color{blue} #1}}

\usepackage{matlab-prettifier}
\usepackage[percent]{overpic}
\usepackage{changepage}
\def\R{\mathbb{R}}

\def\qed{~\relax\ifmmode\hskip2em \Box
 \else\unskip\nobreak\hskip1em \hfill$\Box$
 \fi \newline}

\DeclareMathOperator*{\arginf}{arg\,inf}
\DeclareMathOperator*{\argmin}{arg\,min}
\DeclareMathOperator*{\cp}{ecp}

\usepackage[normalem]{ulem}
\usepackage{soul}%
\definecolor{mygray}{gray}{0.4}

\newtheorem{theorem}{Theorem}
\newtheorem{lemma}{Lemma}
\newtheorem{proposition}{Proposition}

\title{Proving the stability estimates of variational least-squares kernel-based methods}

\author[inst1,inst2,inst5]{Meng Chen}
\affiliation[inst1]{organization={Department of Mathematics, Nanchang University},%Department and Organization
	city={Nanchang},
	country={China}}
\affiliation[inst2]{organization={Institute of Mathematics and Interdisciplinary Sciences, Nanchang University},%Department and Organization
city={Nanchang},
country={China}}
\affiliation[inst5]{organization={Xinjiang Vocational and Technical University of Tianshan},%Department and Organization
	city={Kizilsu Kirgiz Autonomous Prefecture, Xinjiang},
	country={China}}

\author[inst3]{Leevan Ling}
\affiliation[inst3]{organization={Department of Mathematics, Hong Kong Baptist University},%Department and Organization
	country={Kowloon Tong, Hong Kong}}
\author[inst4]{Dongfang Yun}
\affiliation[inst4]{organization={School of Mathematics and Statistics, Central South University},%Department and Organization
	city={Changsha},
	country={China}}

\begin{document}

\begin{frontmatter}

\begin{abstract}
Motivated by the need for the rigorous analysis of the numerical stability of variational least-squares  kernel-based methods for solving second-order elliptic partial differential equations, we provide previously lacking stability inequalities. This fills a significant theoretical gap in the previous work [Comput. Math. Appl. 103 (2021) 1-11], which provided error estimates based on a conjecture on the stability. With the stability estimate now rigorously proven, we complete the theoretical foundations and compare the convergence behavior to the proven rates. Furthermore, we establish another stability inequality involving weighted-discrete norms, and provide a theoretical proof demonstrating that the exact quadrature weights are not necessary for the weighted least-squares kernel-based collocation method to converge.
Our novel theoretical insights are validated by numerical examples, which showcase the relative efficiency and accuracy of these methods on data sets with large mesh ratios. The results confirm our theoretical predictions regarding the performance of variational least-squares kernel-based method, least-squares kernel-based collocation method, and our new weighted least-squares kernel-based collocation method. Most importantly, our results demonstrate that all methods converge at the same rate, validating the convergence theory of weighted least-squares in our proven theories.
\end{abstract}

\end{frontmatter}

\section{Introduction}
{
In the early 1990s, E.J. Kansa \cite{Kansa-Multscatdataappr:90} introduced a method to solve partial differential equations (PDEs) using the meshfree characteristics of radial basis function (RBF), which is generally referred to as kernel methods nowadays. Despite its adoption in various engineering and scientific applications, the theoretical foundations of the Kansa method were not established until well into the following decade.
The initial theoretical backing attempts were unsuccessful until 2001 when Hon and Schaback \cite{Hon+Schaback-unsycollradibasi:01} identified that the original Kansa method could experience failures in rare cases due to the singularity of the underlying linear system. Building on this, the author and his collaborators provided in 2006 \cite{Ling+OpferETAL-Resumeshcolltech:06} the first solvability results for a modified Kansa method that incorporated the idea of ``overtesting''. This approach involved using a larger set of collocation points compared to the trial centers, resulting in an overdetermined system that required defining numerical solutions through a minimization process.

Further analysis by the author on the convergence in maximum norm of this overdetermined kernel collocation formulation \cite{Ling+Schaback-Stabconvunsymesh:08} was published in 2008. In 2016, Schaback published a general framework for convergence proof and demonstrated that the convergence rate in the maximum norm could match that of the interpolant under specific smoothness conditions \cite{Schaback-WellProbhaveUnif:16}.
For computational sake, recent developments have seen the establishment of convergence theories for the least-squares minimizer in the overdetermined Kansa formulation \cite{Cheung+LingETAL-leaskerncollmeth:18}.

Most recently, in 2021, Seyednazari \textit{et al.}  attempted to prove the convergence of a variational least-squares Kansa formulation. However, their published proof was found to be incomplete \cite{Seyednazari+TatariETAL-Errostabestileas:21}, indicating that there are still unresolved issues in fully validating this approach.
}
Throughout the paper, we mainly follow the notations in
\cite{Seyednazari+TatariETAL-Errostabestileas:21} (hereafter referred to as ``the referenced work'') for the sake of consistency, clarity, and to make the comparison of our results with those presented in the referenced work easy for readers.
%\cite{Agmon+DouglisETAL-Estinearbounsolu:64} and the referenced work.
We consider second-order elliptic PDEs in a bounded domain $\Omega\subset\mathbb{R}^d$ subject to the Dirichlet boundary condition
\begin{subequations}\label{eq:bvp}
\begin{align}
\mathcal{L}u &= f, \qquad \text{in} \ \Omega, \label{eq:bvp1} \\
u &= g, \qquad \text{on} \ \partial\Omega, \label{eq:bvp2}
\end{align}
\end{subequations}
for a linear second-order uniformly elliptic operator
\begin{equation*}
\mathcal{L}u \coloneqq \sum_{i,j=1}^d a_{ij}(x)\dfrac{\partial^2 u}{\partial x_i \partial x_j} - \sum_{i=1}^d b_i(x) \dfrac{\partial u}{\partial x_i} - c(x) u,
\end{equation*}
with smooth and bounded coefficients.
This paper addresses the need for a rigorous analysis of the numerical stability estimates of variational least-squares (VLS) kernel-based methods for solving second-order elliptic PDEs.
Our analysis complements the analysis presented in the referenced work, which offered error estimates based on a conjecture {\cite[Eqn.~34]{Seyednazari+TatariETAL-Errostabestileas:21}} to prove the stability estimates.
%Subject to the assumptions on the kernel, which we will outline in an upcoming section,
Our aim is to prove that, for any $0\leq q \leq \tau-2$, there exists a constant $C$ such that the following stability estimate (referenced in Lemma~3.8 of the referenced work) holds:
\begin{equation}\label{eq stab ineq}
C^{-1} h_{}^{2q}\| u^h\|^2_{H^{q+2}(\Omega)} \leq \| \mathcal{L}u^h\|^2_{L^2(\Omega)} + h^{-3} \| u^h\|^2_{L^2(\partial\Omega)},
\end{equation}
for all trial function $u^h \in U_{\Phi,X}$ in the finite dimensional trial spaces $U_{\Phi,X}$ spanned by Sobolev space $H^\tau(\Omega)$ reproducing kernels $\Phi$ centered at the set of discrete data points $X\subset\Omega$ with fill distance $h = h_{X,\Omega}$.
We provide a formal proof for the stability inequality  without using the conjecture. %\CMout{\cite[Eqn.~34]{Seyednazari+TatariETAL-Errostabestileas:21}}.
By providing a rigorous proof of \eqref{eq stab ineq}, we fill a significant gap in the existing literature and provide a comprehensive theoretical foundation for these methods.

In Section~\ref{sec lsvkm}, we introduce necessary notations to ensure the stability result is understandable to readers and review the logical structure of error analysis concerning both continuous and discrete (i.e., collocation) least-squares kernel-based methods. Section~\ref{sec:proof} contains a formal proof of inequality \eqref{eq stab ineq}. We also restate the resulting corollaries about error and condition number estimates, which are immediately valid given {in} our proof.
Next, in Section~\ref{sec WLS},  we prove another stability inequality involving weighted-discrete norms.
This inequality is the key to the convergent analysis of a weighted least-squares (WLS) kernel-based collocation method.
In Section~\ref{sec num exmp}, we compare these proven theoretical results of the various implementations of the method. This comparison provides insight into the relative efficiency and accuracy on   data sets with large mesh ratio.
We conclude in Section~\ref{sec conc} with a discussion of implications and future work.

\section{Assumptions and preliminaries}\label{sec lsvkm}

Given appropriate smoothness and boundedness assumptions on the domain $\Omega$, the differential operator $\mathcal{L}$, and the data $f\in H^{k-2}(\Omega)$ and $g\in H^{k-1/2}(\partial\Omega)$, we assume that the PDE \eqref{eq:bvp} admits a classical solution $u \in H^k(\Omega)$ for some $k\geq 2$ and satisfies a boundary regularity estimate \cite{Jost-Partdiffequa:07} of the form
\begin{equation}\label{eq: reg est}
\|u\|_{H^{j}(\Omega)} \leq C \big(\|f\|_{H^{j-2}(\Omega)} + \|g\|_{H^{j-1/2}(\partial\Omega)}\big)
\qquad\text{ {for $2\leq j\leq k$}}.
\end{equation}
In our previous proof for the error estimates involving least-squares kernel-based collocation methods for solving the PDE in \eqref{eq:bvp}, we made several standard assumptions, one being the boundary regularity in \eqref{eq: reg est}. Other assumptions involved the domain, differential operator, and solution characteristics, as outlined in Assumptions~2.1--2.4 of  \cite{Cheung+LingETAL-leaskerncollmeth:18} (hereafter referred to as ``our prior work'').
In the referenced work focusing on VLS kernel-based methods, the authors assumed the PDE was of Agmon-Douglis-Nirenberg type \cite{Agmon+DouglisETAL-Estinearbounsolu:64}, enabling them to assert the boundary regularity.

We work in some norm-equivalent Hilbert spaces $H^\tau(\Omega)$, with $\tau> d/2$ and $\tau\geq k$ the smoothness of the PDE solution, reproduced by some symmetric positive definite kernel $\Phi$ {with translation-invariance} whose Fourier transform decays algebraically as
\begin{equation}\label{fourierdecay}
C_1 {\left( 1+|\omega|_2^2 \right)}^{-\tau} \leq \widehat{\Phi}(\omega) \leq C_2 {\left( 1+|\omega|_2^2 \right)}^{-\tau}
\qquad \text{for all $\omega\in \mathbb{R}^d$},
\end{equation}
and two positive constants $C_1$ and $C_2$.
This setup allows us to {use} the Mat\'{e}rn kernel, which reproduces Sobolev spaces $H^\tau(\R^d)$, defined as:
\begin{equation}\label{eq:MaternKernel}
    \Phi_\tau(\cdot,x_j)=C_\tau  \|\cdot-x_j\|^{m-d/2}K_{\tau-d/2}(  \|\cdot-x_j\|),
\end{equation}
where $K_{\nu}$ is the modified Bessel function of the second kind  and $C_\tau=2^{1-(\tau-d/2)}/\Gamma(\tau-d/2)$ normalizes the kernel.

Let $X = \left\{ x_1, x_2, \ldots, x_{N_X} \right\}\subset \Omega\cup\partial\Omega$ be a sufficiently dense discrete set (with respect to $\Omega$, $\Phi$, and $\mathcal{L}$) of $N_X$ trial centers. We seek numerical solutions in finite-dimensional trial spaces spanned by translations of $\Phi$:
\begin{equation}\label{eq:Ux}
U_{\Phi,X} \coloneqq \text{span}\left\{ \Phi(\cdot,x_j) : x_j \in X \right\} \subset H^\tau(\Omega).
\end{equation}
The \emph{variational least-squares solution}   in the referenced work is defined by
\begin{equation}\label{eq:ls var sol}
u^h_{{VLS}} \coloneqq \frac{1}{2} \arginf_{v\in U_{\Phi,X}} \left( \|\mathcal{L} v-f\|_{L^2(\Omega)}^2 + h^{-3}\|v-g\|_{L^2(\partial\Omega)}^2 \right),
\end{equation}
whose error estimates were made in terms of the fill distance of $X$:
\begin{equation}\label{filldist}
h =  h_{X} := \sup_{\zeta\in\Omega} \min_{x_j \in X} \|\zeta-x_j\|_{\ell^2(\mathbb{R}^d)}.
\end{equation}
In later theorems, fill distances of other sets of points will be denoted with different subscripts. We also need the separation distance of $X$ given by
\begin{equation*}
	q_X \coloneqq \frac{1}{2} \min_{\substack{x_i,x_j\in X \\ x_i \neq x_j}} \|x_i-x_j\|_{\ell^2(\R^d)},
\end{equation*}
and {the} mesh ratio is $\rho_X :=  h_{}/q_X$. %\magenta{The mesh-dependent weighted least-squares functional, i.e., the objective functional on the right-hand side of \eqref{eq:ls var sol}, is given and studied in \cite{Aziz+KelloggETAL-Leastsquares:85}.  }

Since $u^h_{{VLS}} \in U_{\Phi,X}$, there exist coefficients $\{ c_j \}\subset \R^{N_X}$ such that
$$u^h_{{VLS}}(x) = \sum_{j=1}^{N_X} c_j \Phi(x, x_j).$$
Therefore, the optimization problem \eqref{eq:ls var sol} can be recast in terms of these coefficients as
\begin{eqnarray}
  && u^h_{{VLS}} \coloneqq \frac12 \arginf_{\{ c_j \}\subset \R^{N_X}} \Bigg\{   \int_\Omega \Big( \sum_{j=1}^{N_X} c_j \mathcal{L}\Phi(x, x_j)-f \Big)^2dx
  \\ \nonumber
  &&\qquad\qquad\qquad\qquad\qquad + h^{-3}\int_{\partial\Omega} \Big(\sum_{j=1}^{N_X} c_j \Phi(x, x_j)-g\Big)^2 dx \Bigg\},
  \nonumber
\end{eqnarray}
whose minimizer
{$\alpha:=[c_1,\ldots, c_{N_X} ]^T$}
%$\{ c_j \}\subset \R^{N_X}$
is the solution of the following $N_X\times N_X$ matrix system
\begin{eqnarray}
     &&    \int_\Omega \mathcal{L}\Phi(x, x_l) \Bigg(  {\sum_{j=1}^{N_X}c_j} \mathcal{L}\Phi(x, x_j)-f(x) \Bigg)dx
\label{eq:ls var sol2}
\\
     &&\qquad\qquad  + h^{-3}\int_{\partial\Omega} \Phi(x, x_l)\Bigg(  {\sum_{j=1}^{N_X}  c_j}\Phi(x, x_j)-g(x) \Bigg)dx  = 0,\nonumber
\end{eqnarray}
for $1\leq l \leq N_X$.
Upon proving the stability inequality \eqref{eq stab ineq} (i.e., Lemma 3.8 of the referenced work) in the next section, we will have a theoretically sound error estimate for $u^h_{{VLS}}$ (by Theorem 3.12 and Corollary 3.16 of the referenced work):
\begin{theorem}\label{main thm}
Let  $d\leq 3$ and $\tau\geq k\geq 4$ (hence, $\tau>d/2$). Suppose $\Omega\subset \R^d$ is a bounded domain %Lipschitz
with $C^k$-boundary $\partial\Omega$
satisfying an interior cone condition. Let $\mathcal{L}$ in the PDE \eqref{eq:bvp} be a second-order strongly elliptic operator satisfying boundary regularity estimates in the form of \eqref{eq: reg est}. Also, suppose that the solution $u\in H^k(\Omega)$ to \eqref{eq:bvp} has smoothness order $k$.
Let $\Phi:\Omega\times \Omega\to \R$ be a symmetric positive definite kernel that reproduces the Sobolev space $H^\tau(\Omega)$ with smoothness $\tau$. Let $u^h_{{VLS}}\in U_{\Phi,X}\subset H^\tau(\Omega)$, which belongs to the trial space \eqref{eq:Ux} with quasi-uniform trial centers in $X$ and fill distance $ h_{}$, denote the variational least-squares solution defined in \eqref{eq:ls var sol}.
Then, the 2-norm condition number of the linear matrix system \eqref{eq:ls var sol2} associated with \eqref{eq:ls var sol} is bounded above by $C h^{-4\tau}$. Moreover, the following error estimate holds:
\[
\| u-u^h_{{VLS}}\|_{H^t(\Omega)}\leq C h^{k-t}\| u \|_{H^k(\Omega)}
\qquad \text{for any $0\leq t \leq k$},
\]
where $C$ is a generic constant independent of $u$ and $u^h_{{VLS}}$.
\end{theorem}

Theorem~\ref{main thm} suggests that setting $\tau=k$ has no effect on the error estimate but yields the lowest bound on the condition number. In practice, one usually chooses $\tau$ based on certain numerical considerations, under the assumption that the solution is at least as smooth as the kernel. In other words, we use $k=\tau$ in theory, even though the true solution smoothness could be greater than $\tau$.

A general framework for proving high-order convergence can be found in \cite{Schaback-WellProbhaveUnif:16}. We outline the key elements in the proof for discussions in later sections. In our context, the PDE \eqref{eq:bvp} and numerical methods satisfy the followings:
\begin{itemize}
  \item \textbf{Well-posedness} of the PDE in the sense that any PDE solution can be bounded above by that of the associated data functions under some appropriate norms. The boundary regularity \eqref{eq: reg est} serves this purpose.
  \item \textbf{Trial space} is a finite-dimensional space from which we seek a numerical approximation to the PDE solution. The assumptions made about the regularity/smoothness of the PDE solution determine the choice of the trial space.
  \item \textbf{Stability estimate}  bounds all trial functions from above using a data norm (which can be easily computed and may be either continuous or discrete), usually derived from the well-posedness inequality. Its main role is to {bound} the numerical solution via optimization. Specifically, the objective function in the definition \eqref{eq:ls var sol} of   VLS solutions is essentially (constant multiple of) the right-hand side of the stability inequality. This is typically the most challenging part of designing a convergent numerical method.
       For instance, the desired stability inequality \eqref{eq stab ineq} uses $L^2$-norms on the right-hand side, unlike the Sobolev norms in the boundary regularity. In the referenced work, a Bernstein-type inverse inequality
   \begin{equation}\label{Xinvineq}
     \|u\|_{H^{3/2}(\partial\Omega)}\stackrel{?}{\leq}  {C}  h^{-3/2}\| u\|_{L^2(\partial\Omega)}
     \text{\quad for all $u\in U_{\Phi,X}$}
   \end{equation}
   is used to derive stability inequality \eqref{eq stab ineq} from boundary regularity \eqref{eq: reg est}.
  We emphasize by using $\stackrel{?}{\leq}$ that this is just a conjectured inequality, which lacks theoretical support due to the mismatch between trial centers $X\in \Omega$ and norms on $\partial \Omega$.

  \item \textbf{Consistency estimates} use the approximation power of the trial space  to determine the  error bounds and convergence rates of the numerical method. Since the numerical solution is defined via an optimization problem, we can replace the numerical solution in the objective function with some comparison trial function in the error analysis. This comparison function could be, for instance, the interpolant of the true solution from the trial space, which provides an upper bound. The remainder of the analysis can be conducted using standard approximation theory. {Interested readers are encouraged to refer to the earlier works cited in the introduction for detailed proofs of convergence. The stability estimates provided in this paper are independent of the consistency analysis.}
\end{itemize}

%\subsubsection*{Remark 1}\label{remark1} Under close inspection, the condition $d\leq 3$ in Theorem~\ref{main thm} is unnecessary as this condition arises when the authors of the referenced work try to achieve results for lower-order Sobolev norms from the above framework for higher-orders. Their $L^2$-error estimate, which is used in the interpolation theorem of Sobolev spaces, was proven with different arguments.  \qed

%\subsubsection*{Remark 1}\label{remark1} Under close inspection, {the condition $d\leq 3$ mentioned  in the referenced work  is unnecessary for the VLS solution in Theorem~\ref{main thm}.} This condition arises when the authors of the referenced work try to achieve results for lower-order Sobolev norms from the above framework for higher-orders.
%{Their interpolation theorem of Sobolev spaces between  $L^2$-error  and   $H^2$ for $d\leq 3$
%can indeed be extended to $H^\nu$ for $d>3$ with $\nu>2$.} \qed

Evaluating the integrals in the matrix system \eqref{eq:ls var sol2} exactly can be a challenging task due to the complexity of the involved functions and the potential high-dimensionality of the irregularly-shaped domain.
To convert the matrix system in \eqref{eq:ls var sol2} into a computable scheme, we use sets of \emph{quadrature nodes} in the domain and on the boundary
\begin{equation}\label{YZ}
Y = \left\{ y_1, y_2, \cdots, y_{N_Y} \right\}\subset \Omega\cup\partial\Omega
\text{\quad and \quad}
Z = \left\{ z_1, z_2, \cdots, z_{N_Z} \right\}\subset \partial\Omega,
\end{equation}
with positive \emph{quadrature weights}, which {are} stored in a diagonal matrix for convenience,
\begin{equation}\label{W}
W= \text{Diag}\left[ \omega_1, \omega_2, \cdots, \omega_{N_Y}, \omega_{ {N_Y+1}},\cdots, \omega_{N_Y+N_Z}  \right]
\end{equation}
to approximate the integrals in $\Omega$ and on $\partial \Omega$.
%Note that Theorem~\ref{main thm} does not consider i {}ntegration error.
Then, an approximate version of {the} matrix system \eqref{eq:ls var sol2} can be written as
\begin{equation} \label{VLS mtx sys}
    \begin{bmatrix}
        [\mathcal{L}\Phi](Y,X) \\
        h^{ {-\frac32}}[\Phi](Z,X)
    \end{bmatrix}^T
     W
    \begin{bmatrix}
        [\mathcal{L}\Phi](Y,X) \\
        h^{ {-\frac32}}[\Phi](Z,X)
    \end{bmatrix} \alpha
    =
    \begin{bmatrix}
        [\mathcal{L}\Phi](Y,X) \\
        h^{ {-\frac32}}[\Phi](Z,X)
    \end{bmatrix}^T
    W
    \begin{bmatrix}
        f(Y) \\
        h^{ {-\frac32}}g(Z)
    \end{bmatrix},
\end{equation}
which is nothing but the least-squares solution to the following $(N_Y+N_Z)\times N_X$ overdetermined matrix system
{for an unknown coefficient vector $\alpha\in \R^{N_X}$:}
\begin{equation}\label{VLS mtx sys LS}
    W ^{1/2}
    \begin{bmatrix}
        [\mathcal{L}\Phi](Y,X) \\
        \vartheta[\Phi](Z,X)
    \end{bmatrix} \alpha
    =
    {W} ^{1/2}
    \begin{bmatrix}
        f(Y) \\
        \vartheta g(Z)
    \end{bmatrix},
    \text{\quad with $\vartheta=h^{ {-\frac32}}$.}
\end{equation}
If all data points have fill distances of the same magnitude and the quadrature weights  are constant, i.e., ${W} = \text{Vol}(\Omega)N_Y^{-1} I$, say in the case of Monte Carlo integration, then \eqref{VLS mtx sys LS} is exactly the kernel-based least-squares collocation (LSC) systems with oversampling in our prior work \cite{Cheung+LingETAL-leaskerncollmeth:18}. This close proximity motivates proving \eqref{eq stab ineq} by the stability inequality in the discrete counterpart.

\section{Proof of stability estimates}\label{sec:proof}

The theoretical work presented below makes use of our previously established theories on least-squares collocation solutions to \eqref{eq:bvp} from our prior work \cite{Cheung+LingETAL-leaskerncollmeth:18} to prove the target stability inequality \eqref{eq stab ineq}. We now let the points in the sets \eqref{YZ} play the role of collocation points. Furthermore, we also define the discrete norm for a function $u$ with the set $P$ of $N_P$ points by
\begin{equation}\label{Y-norm}
  \|u\|_{P}^2 :=  \|u(P)\|_{\ell^2(\R^{N_P})} ^2 = \sum_{p_i \in P} \big| u(p_i) \big|^2.
\end{equation}
Then, we can use $Y$ and $Z$ to define discrete norms. For completeness, the \emph{least-squares collocation}  solution  (\cite[Theorem 2.6]{Cheung+LingETAL-leaskerncollmeth:18}  with $\theta=2$ and $\nu=q+2$) is defined as
\begin{equation}\label{eq:H2 sol}
u^h_{{LSC}} \coloneqq \frac{1}{2} \arginf_{v\in U_{\Phi,X}}
\left( \|\mathcal{L} v-f\|_{Y}^2 + h_Y^{-d+2q} h_Z^{ { d-4-2q}}\|v-g\|_{Z}^2 \right),
\end{equation}
whose objective function stems from the following  stability estimate involving discrete norms (\cite[Lemma 3.3]{Cheung+LingETAL-leaskerncollmeth:18} with $\nu=q+2$):

\begin{proposition}\label{my stability thm}
    Under the assumptions of Theorem~\ref{main thm}, for any sets of collocation points $Y\subset\Omega\cup \partial\Omega$ and $Z\subset\partial\Omega$ satisfying the denseness requirement
    \begin{equation}\label{denseness}
      C \left(
      h_Y^{\tau-q-2} + h_Z^{\tau-q-2}
      \right) h_{}^{-\tau+q+2}
      <\frac12,
    \end{equation} the stability inequality
    \begin{equation}\label{my stab est}
    	\| u^h \|_{H^{q+2}(\Omega)}^2
    \leq
    C \left( h_Y^{d-2q} \| \mathcal{L}u^h \|_Y^2 + h_Z^{d-4-2q} \|u^h\|_Z^2 \right)
    ,\quad 0\leq q \leq \tau-2,
    \end{equation}
    holds for all trial functions $u^h\in U_{\Phi,X}\subset H^\tau(\Omega)$, with some constant $C$ that only depends on $q$, $\Omega$, $\Phi$ and $\mathcal{L}$.
\end{proposition}

To prove the discrete stability estimate in \eqref{my stab est}, we require theoretical tools such as sampling inequalities \cite{Arcangeli+SilanesETAL-ExtesampineqSobo:12} and inverse inequalities. These inequalities, which differ from the conjecture \eqref{Xinvineq}, were proven in our prior work. We direct interested readers to the original articles for more details.
It is easy to verify that the collocation matrix in this context corresponds to the overdetermined system in the form of \eqref{VLS mtx sys LS} with unity quadrature weight. The trial function that minimizes \eqref{eq:H2 sol} has expansion coefficients ${\{ c_j \}}\subset \R^{N_X}$ that solve \eqref{VLS mtx sys LS} in the least-squares sense, and \eqref{VLS mtx sys} is the associated normal equation.

Now, we rewrite the  inequality \eqref{my stab est} as
\begin{equation}\nonumber %label{my stab est 2}
C^{-1} \| u^h \|_{H^{q+2}(\Omega)}^2 \leq
    h_Y^{-2q}  h_Y^d \| \mathcal{L} u^h \|_Y^2 + h_Z^{-2q} h_Z^{-3}h_Z^{d-1} \|u^h\|_Z^2 ,
\end{equation}
which matches the form of our targeted inequality in \eqref{eq stab ineq}. Note that \eqref{eq stab ineq} can be proven if, for any given trial function $u^h$, we can select sets $Y$ and $Z$ so that
\begin{equation}\label{what to prove}
   \text{$h_Y^d \| \mathcal{L}u^h \|_Y^2 \leq  \| \mathcal{L} u^h\|^2_{L^2(\Omega)}$,
   \qquad and \qquad
    $h_Z^{d-1} \|u^h\|_Z^2\leq \| u^h\|^2_{L^2(\partial\Omega)}$}
\end{equation}
hold, and at the same time, subject to a constraint on fill distances
$$h_Y, h_Z = \mathcal{O}( h_{})$$
to allow factoring the term $ h_{}^{2q}$ to the left-hand side. We  present two lemmas to prove \eqref{what to prove} before the formal proof of the stability inequality \eqref{eq stab ineq}. 
%\blue{The one (Lemma~\ref{lemY}) and another boundary version (Lemma~\ref{lemZ}) are derived by $h_P$, instead of $n_P$ in analogies in \cite[Theorem 7]{wenzel2024sharpinversestatementskernel}.}

\begin{lemma}\label{lemY}
 Let $\Omega \subset \mathbb{R}^d$ be a compact domain.
Let $v:\Omega \rightarrow \mathbb{R}$ be any Riemann square-integrable function. For any given $C>0$ and $h>0$, there exits a discrete set of points $P\subset \overline{\Omega}$ with fill distance $ {\frac12}Ch \leq h_P  \leq C h$ satisfying
   $h_P^d \| v \|_P^2 \leq  \|v\|^2_{L^2(\Omega)}$.
\end{lemma}
\begin{proof}
%\proof
We present a constructive proof by first picking a sufficiently large rectangular domain $\mathcal{R}:=[-l,l]^d\supset\Omega$ with some $l\in \delta\mathbb{N} :=  { d^{-1/2}Ch}\mathbb{N}$. Then, $\mathcal{R}$ can be partitioned by equally sized square subdomains with volume $\delta^d$; we call this partition $\mathcal{P}$. We pick an initial set of (possibly non-distinct) points $Q\subset\mathcal{R}$ that can achieve the lower Riemann sum $\underline{R}_{\mathcal{P}}$ on $\mathcal{P}$. That is,
\[
    \underline{R}_{\mathcal{P}}(\chi_\Omega v^2 ) = h_Q^d \| v \|_Q^2
    \leq \int_\mathcal{R} \chi_\Omega v^2\,dx
    = \int_\Omega v^2\,dx = \|v\|^2_{L^2(\Omega)},
\]
where $\chi_\Omega$ is the characteristic function for $\Omega$.
Note that the points in $Q$ may not be distinct or even in the domain; we define the set of distinct points $P=\text{unique}(Q\cap\overline{\Omega})$ by removing points outside of $\overline{\Omega}$ and duplicates in $Q$. By doing so, we must have $h_P=h_Q$ and $\|v\|_P^2 \leq \|v\|_Q^2$.

We conclude the proof by observing two critical scenarios. The first scenario arises when the largest possible fill distance of $P\subset \overline{\Omega}\subset \mathbb{R}^2$ ($\mathbb{R}^3$) occurs. This happens in a local region where two of the four (eight) points in a 4-square quadrant (8-cube cuboid) region are located in opposite corners. This is illustrated in Figure~\ref{fig:Lem1}. In general, we have ${(2h_P)}^2 \leq d(2\delta)^2$ for dimension $d$.
In the second scenario, the smallest fill distance, $ {d\delta^2\leq (2h_P)}^2 $, occurs when all data points are situated at the center of their respective squares (cubes). Therefore, it follows that  {$\frac{1}{4}C^2h^2 \leq h_P^2 \leq C^2h^2$} as required.
%We complete the proof by noting that the largest possible fill distance of $P\subset \overline{\Omega}\subset \mathbb{R}^2$ ($\mathbb{R}^3$) occurs when, in some local region, two of the four (eight) points in a 4-square quadrant (8-cube cuboid) region are located in the opposite corners, see Figure~\ref{fig:Lem1}, i.e., $(2h_P)^2 \leq d(2\delta)^2$ for dimension $d$. Similarly, the smallest fill distance is $(2h_P)^2 \leq d\delta^2$ when all data points sit at the center of their respective squares (cubes). That is, $\frac{1}{4}Ch \leq h_P \leq Ch$ as desired.
%\qed
\end{proof}

In the course of our study, it has come to our attention that similar themes and theoretical results have been explored in an unpublished work \cite{wenzel2024sharpinversestatementskernel}. Specifically, the norm equivalency there were built upon/for nested quasiuniform points generated by a geometrically greedy algorithm \cite{Marchi+SchabackETAL-Neardatapoinloca:05}.

\begin{figure}
  \centering
  \begin{overpic}
	[width=.3\hsize,clip=true,tics=10]{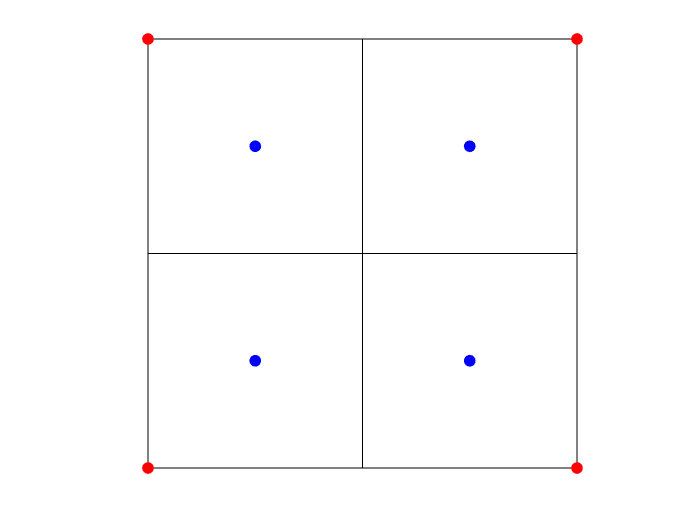}
    \put(85,52) {\scriptsize $\delta$}
    \put(52,38.8){\color{red}\vector(1,1){30}}
    \put(40,75) {\scriptsize \red{Largest possible $h_P$}}
    \put(52,38.5){\color{blue}\vector(1,-1){15}}
    \put(70,25) {\scriptsize \blue{Smallest}}
  \end{overpic}
  \includegraphics[width=.35\hsize]{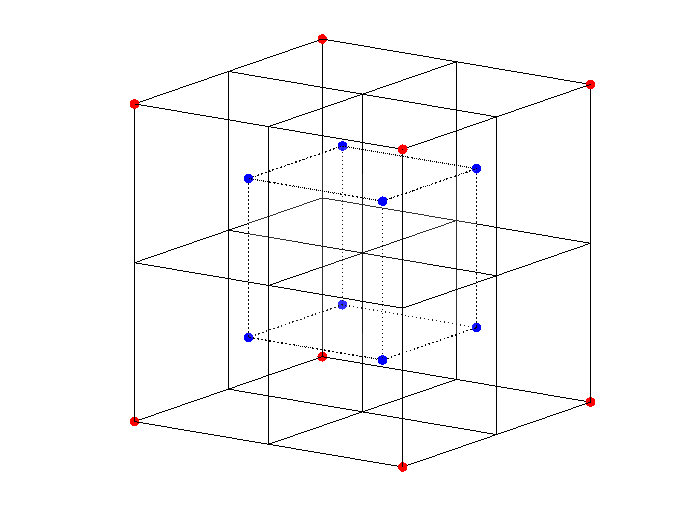}
  \caption{
   An illustration of the distribution of data points contributing to the lower Riemann sum based on a partition with equally sized subdoamins with edge length $\delta$. The red points showcase a distribution that maximizes the fill distance when some points are located at opposite corners of one quadrant/cuboid region.
The blue points show a distribution that minimizes the fill distance when all the points are in the centers of all quadrant/cuboid regions.
  }\label{fig:Lem1}
\end{figure}
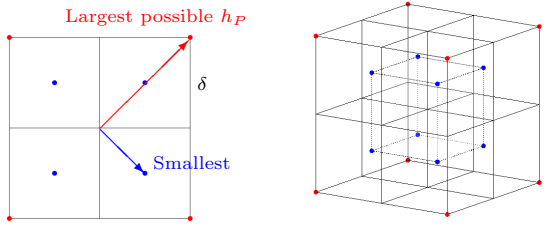

\begin{lemma}\label{lemZ}
 Let $\partial\Omega \subset \mathbb{R}^d$ be a codimension-1, $C^1$-smooth, connected and compact manifold. Let $v:\partial\Omega \rightarrow \mathbb{R}$ be any Riemann square-integrable function.
Denote by $r_{\partial\Omega}>0$ a constant dependent on $\partial\Omega$ for which the constant-along-normal extension is well-defined in the narrow band domain of radius $r<r_{\partial\Omega}$.
For any given $C>0$ and $h>0$ sufficiently small with respect to $r_{\partial\Omega}$ and other generic constants, there exists a discrete set of points $P\subset\partial\Omega$ with fill distance ${\frac12}Ch \leq h_P  \leq C h$ satisfying
{$h_P^{d-1} \|v\|_P^2 \leq \|v\|^2_{L^2(\partial\Omega)}$}.
\end{lemma}
%\proof
\begin{proof}
We prove this boundary version by embedding and application of Lemma~\ref{lemY}.
First, we pick a constant $r<r_{\partial\Omega}$ so that the Euclidean closest point function
\[
    \cp(x) := \argmin_{y\in\partial\Omega} |x-y|  \quad \text{for $x\in B(\partial\Omega)$}
\]
is $C^0$-continuous  in the $r$-narrow band domain:
\[
    B(\partial\Omega):= \big\{ x\in\R^d:\, \|x-y\|_2\leq r < r_{\partial\Omega}, \quad \text{for some $y\in\partial\Omega$} \big\}.
\]
This is a relaxed version of the $\cp$-Calculus in \cite{Maerz+Macdonald-CALCSURFWITHGENE:12} as we do not require any differentiability in this proof.
Using \cite[Lemma 3.1]{Cheung+Ling-Kernembemethconv:18} to relate the norms of surface functions and their constant-along-normal extensions defined by $x := y + t\vec{n}(y): \partial\Omega \to \Gamma_t$ for $y \in \partial\Omega$ and $t \in [0, r]$, there exists a constant $C_{\partial\Omega}$ that depends only on $\partial\Omega$ such that
\begin{eqnarray}\label{ecp1}
    \|v\circ\cp\|^2_{L^2(B(\partial\Omega))}
    :=\int_{-r}^{r}\|v\circ\cp\|^2_{L^2 (\Gamma_t  )}dt
    \leq C_{\partial\Omega} r \|v\|^2_{L^2( \partial \Omega) }
\end{eqnarray}
holds for all $v\in L^2( \partial \Omega)$.

For any sufficiently small $h$, let us focus on some rectangular region $\mathcal{R}$ containing the band $B(\partial\Omega)$. We partition $\mathcal{R}$ as in the previous proof with $r=C^dh/({2}^d C_{\partial\Omega})< r_{\partial\Omega}$.
Now, we apply Lemma~\ref{lemY} with the given constants $C,h>0$ to the constant-along-normal extension $v\circ\cp :B(\partial\Omega) \to \R$  of any given $v\in L^2( \partial \Omega)$. By  \eqref{ecp1}, this gives us an initial set of distinct points  $Q\subset B(\partial\Omega)$ so that
\begin{eqnarray}
      h_Q^d \| v\circ\cp  \|_Q^2
     &\leq&  \|v\circ\cp\|^2_{L^2(B(\partial\Omega))}
     \nonumber\\
     &\leq&  C_{\partial\Omega} r \|v\|^2_{L^2( \partial \Omega) }
     \nonumber\\
     &=&  \frac{C^d}{{2}^d} h \|v\|^2_{L^2( \partial \Omega) }.
     \label{eq:vecp}
\end{eqnarray}
%\CMout{From  Lemma \ref{lemY}, we know that {the} range of possible fill distance is ${\frac12}Ch\leq h_Q\leq  Ch$.  The above inequality combined with  \eqref{ecp1} yields}
%\begin{equation*}      \label{Qineq0}
%   \CMout{    \|v\|^2_{L^2( \partial \Omega) }
%       \geq
%       \frac1{h}\Big({\frac2C} h_Q\Big)^{d} \| v\|_Q^2
%       \geq h^{d-1}\| v\|_Q^2.}
%\end{equation*}
We define the desired boundary point set as $P=\text{unique}(\cp(Q))\subset\partial\Omega$. Its fill distance $h_P$ is defined similarly to \eqref{filldist}, but the supremum is taken over a smaller set, $\partial \Omega$. Consequently, we must have $h_P\leq h_Q$.
From the previous proof, we understand that the minimum fill distance occurs only for certain point distributions. Any alteration in the points will result in a change from this minimum fill distance.
Based on these observations, we conclude that the fill distance $h_P$ can only assume values within the range of $h_Q$. In particular, from  Lemma \ref{lemY}, we know that {the} range of possible fill distance is ${\frac12}Ch\leq h_P\leq h_Q\leq  Ch$. Continuing to work on \eqref{eq:vecp} yields the following results:
\begin{eqnarray}      \label{Qineq}
	\|v\|^2_{L^2( \partial \Omega) }&\geq&
		\frac1{h}\Big({\frac2C} h_P\Big)^{d} \| {v\circ\cp} \|_Q^2\\\nonumber
		& \geq&
h^{d-1}	\| {v\circ\cp} \|_P^2\\\nonumber
&=&h^{d-1}	 \| v \|_P^2,
\end{eqnarray}
and the proof is completed. 
\end{proof}%\qed

\subsubsection*{Remark 1}
In the proof in Lemma \ref{lemZ}, the fill distance $h_{P}$ is implicitly characterized by Euclidean distances. For sufficiently small $ h_{}<r_{\partial\Omega}$, then we can pick a narrow band domain to have radius $r< h_{}$. From the proof of our earlier work \cite[Lemma 3.2]{Chen+Ling-Extrmeshcollmeth:20}, we know the two fill distances are equivalent, i.e.,
$   C_1 h_P  \leq h_{P,\partial\Omega} \leq C_2 h_P $
holds for some $C_1,C_2$ depending on the diameter of $\partial\Omega$. Using $h_{P,\partial\Omega}$ in Lemma \ref{lemZ} only influences the generic constants in the proof, as the geodesic fill distance $h_{P,\partial\Omega}$ is proportional to $h_P$. \qed

Through Lemmas \ref{lemY} and \ref{lemZ}, we have confirmed the existence of sets $Y$ and $Z$ that satisfy the conditions of \eqref{what to prove} and the fill distance constraints $h_Y, h_Z = \mathcal{O}( h_{})$. With these conditions met, we can now apply these findings to establish the stability inequality \eqref{eq stab ineq} of the main theorem.

\begin{theorem}\label{thm stab}
Under the assumptions of Theorem \ref{main thm}, the stability inequality in \eqref{eq stab ineq}
  \[
    C^{-1}h^{2q}\| u^h\|^2_{H^{q+2}(\Omega)} \leq
  \| \mathcal{L} u^h\|^2_{L^2(\Omega)}
   + h^{-3} \| u^h\|^2_{L^2(\partial\Omega)},
   \quad 0\leq q \leq \tau-2,
  \]
  holds  for all trial functions $u^h\in U_{\Phi,X}\subset H^\tau(\Omega)$ {that are Riemann square-integrable} with some constant $C$ that only depends on $q$, $\Omega$, $\Phi$ and $\mathcal{L}$.
\end{theorem}
\begin{proof}
%\proof
We denote $C_{\ref{denseness}}$ as the generic constant in \eqref{denseness}, with its exact value being non-essential for this proof. We apply Lemma~\ref{lemY} and Lemma~\ref{lemZ} with
{$C=(5C_{\ref{denseness}})^{-1/(\tau-q-2)}$}. This enables us to select sets of collocation points $Y\subset\overline{\Omega}$ and $Z\subset\partial\Omega$ that satisfy the denseness requirement \eqref{denseness}. Consequently, this ensures that the discrete stability estimate \eqref{my stab est} holds. These sets $Y$ and $Z$ also meet the condition that $h_Y=\mathcal{O}( h_{}) = h_Z$, which can be used to provide lower estimate to $L^2$-norms in \eqref{what to prove}. Thus, the proof is complete.
\end{proof}%\qed
{Note that we simply adopt the theory for LSC method to prove Theorem~\ref{thm stab}, which does not contain any denseness requirement as there are no collocation or quadrature points in the statement.
}

\section{Implications to least-squares collocation methods}\label{sec WLS}

Before we move on, let's clarify some potential confusion around the different ``weights'' used in our context so far. The  ``quadrature weights'' are the coefficients applied to function evaluations at quadrature points to approximate integrals. In our prior work, we also called the   LSC solutions in \eqref{eq:H2 sol} ``weighted least-squares'' solutions. In this context, ``weighted'' refers  to the fill distances-related boundary scaling factor $\vartheta = (h_Y^{-d+2q} h_Z^{{d-4-2q}})^{1/2}$ inside the objective function, which serves to balance the scale differences between the differential operator and boundary conditions. Finally, the ``weight matrix'' in the WLS solutions {\eqref{eq:WLS}, which is defined later in page~\pageref{eq:WLS},} refers to the matrix that is multiplied with the overdetermined system. This system already includes the theoretical scaling factor. This multiplication is done prior to solving the system using a least-squares approach.

%least-squares variational solution \eqref{eq:ls var sol}
%    - convergence estimates available from $L^2$ to $H^\tau$.
%    - integrations are needed; convergence proof do not take quadrature error into consideration
%    - quadrature weight naturally define a data point-dependent weight matrix for the weighted least-squares in \eqref{VLS mtx sys LS}, which is equivalent to the resultant system \eqref{VLS mtx sys}
%    - there is no restriction on quadrature nodes, besides accuracy concern
%v.s.
%least-squares collocation solutions \eqref{eq:H2 sol}
%    - convergence estimates available from $H^{\lceil d/2 \rceil}$ to $H^\tau$, i.e., in dimension $d\leq3$, as in Theorem~\ref{main thm},  this is $H^2$-convergence.
%    - no integration, hence no quadrature error
%    - theory do not allow the use of weight matrix for the weighted least-squares in \eqref{VLS mtx sys LS}, which seems problematic when using collocation points with large mesh ratio (i.e., less uniform intuitively).
%    - collocation points must fulfil denseness requirement \eqref{denseness}

We will now summarize by comparing the VLS solutions \eqref{eq:ls var sol} and the LSC solutions \eqref{eq:H2 sol}. This comparison reveals that each approach has its pros and cons:
\begin{itemize}
  \item The VLS solutions provide  broader convergence estimates, ranging from $L^2$ to $H^\tau$. In contrast, the LSC solutions  offer  estimates from $H^{\lceil d/2 \rceil}$ to $H^\tau$, which translate to $H^2$-convergence in dimensions $d\leq3$, c.f. assumption in Theorem~\ref{main thm}.  We observe similar numerical performance as reported in \cite{Cheung+LingETAL-leaskerncollmeth:18}, where the $L^2$ errors show two additional orders of convergence compared to the $H^2$ errors.
%\red{\emph{Comment: Page 9, lines 308-327. For the first bullet point, is this difference only theoretical? Do the authors observe convergence rates for $H^t, t<d/2$ in the collocation solutions? Adding a discussion would be more clear. Bullet point 4 is not clear to me in its current form.}}
  \item The VLS solutions require   integrations, and its convergence proofs do not account for quadrature error. The exponentially decaying local Lagrange functions \cite{Fuselier+HangelbroekETAL-Locabasekernspac:13} (applicable to  the Mat\'{e}rn kernel and more) are efficient tools for evaluating boundary integrals.
      Conversely, the LSC solutions  do  not involve integration and are thus free from quadrature error.
  \item The {approximation to} VLS  solutions imposes  no restrictions on quadrature nodes
  {except for accuracy concern.}
  However, the collocation points in the LSC solutions must satisfy the trial center-dependent denseness requirement \eqref{denseness}. Locally refining collocation points, without an appropriate weight matrix, often does not improve the accuracy of collocation solutions.
  \item In terms of the weight matrix {$W$} in approximating \eqref{VLS mtx sys LS}, the VLS  solutions naturally employ  quadrature weights to define a data point-dependent matrix for the WLS.
      %Note that the weighted least-squares sets $W=I$ and depends on some boundary weight $\vartheta$ in \eqref{VLS mtx sys LS}.
      The theory in \cite{{Cheung+LingETAL-leaskerncollmeth:18}} regarding  the LSC solutions does not permit  the additional term $W$, which could pose accuracy issues when dealing with collocation points with a large mesh ratio.
      {In the second half of this paper, we prove that this concern is unfounded.}
\end{itemize}

Our next goal is to establish another weighted-discrete stability estimate by finding an additional upper bound for the proven stability inequality in \eqref{eq stab ineq}. This would lead to a numerical scheme where the associated numerical solution is defined by the weighted least-squares formulation in \eqref{VLS mtx sys LS}. Ideally, the conditions on the weight matrix could accommodate easy-to-compute alternatives beyond  merely quadrature weights. This is particularly relevant given the difficulty of determining quadrature weights on irregular domains in higher dimensions. However, the challenge lies in ensuring that any fixed discrete overestimate to an integral holds for all possible trial functions.

\subsection{A uniform bound for $L^2(\Omega)$-norm in the trial space} %
Here is a quick recap of what is discussed in \cite[Theorem 15]{Santin+Schaback-Appreigekernspac:16} for continuous and symmetric  positive definite kernel $K$  on a compact set $\Omega$.
Mercer's theorem ensures the existence of eigenvalues
$\lambda_1\geq\lambda_2\geq\cdots>0$ with a decay rate
\[
   \sqrt{ \lambda_{j+1}  } < C j^{-\tau/d},
\]
{where $C$ is only dependent on $K$  and $\Omega$, }and the eigenfunctions $\phi_j$ that are orthonormal in $L^2(\Omega)$ and orthogonal in its native space $\mathcal{N}(\Omega)$
with $\|\phi_j\|_{\mathcal{N}(\Omega)} =\lambda_j^{-1}$
{for any $j$ \cite[Theorem 2]{Santin+Schaback-Appreigekernspac:16}} such that
\begin{equation}\label{mercer}
  K(x,y) = \sum_{j=1}^\infty \lambda_j \phi_j(x)\phi_j(y)
  \qquad x,y\in\Omega.
\end{equation}

Let $u \in U_{K,X}$ be defined as in \eqref{eq:Ux} with $K$ in place of kernel $\Phi$. By definition of $U_{K,X}$, express $u$ as a linear combination of the kernel $K(\cdot, x_i)$  centered at $x_i\in X$, i.e.,
$$
    u(x) = \sum_{i=1}^{N_X} a_i K(x, x_i)
    %\red{= \sum_{i=1}^{N_X} a_i \sum_{j=1}^\infty \lambda_j \phi_j(x)\phi_j(x_i) }
$$
for some coefficients $a_i$. We introduce  coefficients
$\displaystyle b_j = \sum_{i=1}^{N_X} a_i \phi_j(x_i)$ to simplify this expression for $u$:
$$u(x) = \sum_{j=1}^\infty b_j {\lambda_j }\phi_j(x).$$
Next, we compute the $L^2$-norm square of $u$. Using the orthonormality of the $\phi_j$, this is given as
\begin{equation}\label{L2 int}
\|u\|_{L^2(\Omega)}^2=
\int_\Omega |u(x)|^2\,dx = \sum_{j=1}^\infty   b_j^2 \lambda_j^2
= {b}^T \underbrace{\big[\text{diag}( {\lambda_1 }^2, {\lambda_2 }^2,\cdots)\big]}_{=: D_\lambda} {b},
\end{equation}
where ${b} = [b_1, b_2, \ldots]^T$ is the $u$-dependent coefficient vector with respect to eigenfunctions.

On the quadrature side, for nodes in $y_k\in Y$ with positive quadrature weights in $\omega_k \in {W}$ for $k=1,\ldots,N_Y$, we expand the nodal values of $u$ in the quadrature formula using Mercer's expansion as follows
\begin{align}
\sum_{k=1}^{N_Y} \omega_k |u(y_k)|^2
&= \sum_{k=1}^{N_Y} \omega_k \left(\sum_{j=1}^\infty b_j  {\lambda_j} \phi_j(y_k)\right)^2 \nonumber\\
&=\sum\limits_{j=1}^\infty \sum\limits_{l=1}^\infty b_j b_l
\left( \sum\limits_{k=1}^{N_Y} \omega_k  {\lambda_j}\phi_j(y_k) {\lambda_l}\phi_l(y_k) \right)  = {b}^T {G} {b}.\label{quad}
\end{align}
The $jl$-entries of the matrix ${G}$ were defined by:
\begin{equation}\label{W_ij}
[{G}]_{jl} =
     \sum\limits_{k=1}^{N_Y}\omega_k  {\lambda_j }\phi_j(y_k)  {\lambda_l }\phi_l(y_k),
\end{equation}
which is a Gram matrix and must be semi-positive definite.

{
The following theorem is crucial for understanding the role of norm equivalency in implementing the necessary numerical integration to find $u^h_{VLS}$. It leads to a conclusion that high-order quadrature rules are not essential to match the convergence order of the kernel method.
}

\begin{theorem}\label{thm oe}
Let $K$ be a continuous and positive definite kernel on a compact set $\Omega$. Define a finite-dimensional trial space $U_{K,X}$ of $K$ on some set of trial centers  $X\subset\Omega$.
Then, for any sufficiently dense set, with respect to $K$ and $\Omega$ but not $X$,  of quasi-uniform points $P = \{p_1,\ldots,p_{N}\}\subset\Omega$, there exists
a set of zeroth-order quadrature weights, $\{\widetilde{\omega}_1,\ldots,\widetilde{\omega}_{N_P}\}\in \widetilde{{W}}$, that are not connected to any convergent quadrature rule but still can be used to under- and overestimate the $L^2(\Omega)$-norm of all trial functions. That is, the following inequality holds:
\[
    \frac14 \| u \|_{P,\widetilde{{W}}}^2 \leq
    \|u\|_{L^2(\Omega)}^2
    \leq 4 \| u \|_{P,\widetilde{{W}}}^2
    := 4 \sum_{k=1}^{N} \widetilde{\omega}_k |u(p_k)|^2
    \text{\qquad for all $u \in U_{K,X}$.}
\]
\end{theorem}
\begin{proof}
%\proof
This inequality can be reformulated in terms of coefficients $b$ using equations \eqref{L2 int} and \eqref{quad} as
\[
    \frac14 {b}^T \widetilde{G} {b} \leq  b^T D_\lambda b \leq 4 {b}^T \widetilde{G} {b}
    \text{\qquad for all $b$,}
\]
{where $\widetilde{G}$ is the Gram matrix associated with some yet-to-be-determined zeroth-order quadrature weights $\widetilde{{W}}$.}
%\red{\emph{Comment: Page 10,11, Theorem 4.1. The proof looks intuitive. ''Thus, we can select a large enough $N_P$ such that W is close to $D_\lambda$. Here, closeness is defined such that the j-th eigenvalue of $W, \lambda_j (W)$, lies within the interval $[0.5 \lambda_j (D_\lambda ), 2\lambda_j (D_\lambda)]$" Can such closeness be achievable? This seems unclear from the current argument (simple bound leads to additive errors $D_\lambda = W + \epsilon$, but here, the author used the multiplicative error bound $0.5W \leq D_\lambda \leq 2W$, which is stronger and needs more justifications) }}
We first note that each element ${G}_{jl}$ of {the} matrix ${G}$ is obtained by summing the products of quadrature weights $\omega_k$ and the corresponding function values, over all quadrature points $p_k \in P$. This sum forms an approximation to the integral over $\Omega$.
Given the orthogonality of eigenfunctions, the matrix ${G}$ tends towards $D_\lambda$ as $h_P\to 0$ (or $N_P\to \infty$, provided $P$ is quasi-uniformly refined, i.e., $h_P \leq C N_P^{-1/d}$), and ${W}$ represents any quadrature weights linked to a convergent quadrature rule.
This limit argument is independent of $X$.
The class of integrands under current consideration that are smooth, decay with respect to indices $j,l$, and have a bounded $H^\tau(\Omega)$-norm, i.e., $\|\lambda_j\phi_j\|_{\mathcal{N}(\Omega)}  \leq 1$.
Thus, we can select a large enough $N_P$ such that ${G}$ is close to $D_\lambda$. Here, closeness is defined such that the $j$-th eigenvalue of ${G}$, $\lambda_j({G})$, lies within the interval $[\frac12\lambda_j(D_\lambda),2\lambda_j(D_\lambda)]$.

We now perform a second perturbation of ${G}$ by altering quadrature weights ${W}$ to some nearby values $\widetilde{{W}}$, resulting in a new {Gram matrix $\widetilde{{G}}$}. We keep the perturbation small so that this new matrix has the property that its $j$-th eigenvalue, denoted as $\lambda_j(\widetilde{{G}})$, lies within $[\frac12\lambda_j({G}),2\lambda_j({G})]$ or $[\frac14\lambda_j(D_\lambda),4\lambda_j(D_\lambda)]$.
%This shows
{By considering any basis formed by the eigenvectors of $\widetilde{{G}}$, we show that
\[
     \frac{1}{4}{b^TD_\lambda b} \leq {b^T\widetilde{G} b} \leq 4 {b^TD_\lambda b}
    \text{\qquad for all $b$.}
\]
Given the symmetry of this inequality, one can reverse the roles of $b^TD_\lambda b$ and $b^T\widetilde{G} b$ and} our claim is proven.
%\qed
\end{proof}

\subsubsection*{Remark 2}
{The above analysis is applicable to a wider class of {symmetric} positive definite kernels. However, to align with the convergence proof in other sections, we primarily focus on kernels that reproduce Sobolev spaces, such as the Matérn kernels}
%All discussions in this section hold for Mat\'{e}rn kernels
$\Phi$   of smoothness order $\tau > d/2$, whose native space is $H^\tau(\Omega)$. Moreover, it is also applicable to the restricted Matérn kernels $\Phi_{|_{\partial\Omega \times \partial\Omega}}$ on the boundary \cite{Fuselier+Wright-ScatDataInteEmbe:12}, whose native space is $H^{\tau-1/2}(\partial\Omega)$. \qed

\subsection{Weighted-LS kernel-based collocation method is convergent}\label{sec WLSconv}

{Our final theoretical goal is to demonstrate that the VLS solver can be implemented without using exact quadrature weights, which becomes a WLS problem.}
In this section, we walk the reader through the main ideas of the construction of a convergent method for the WLS solutions in \eqref{VLS mtx sys LS}. Using the weight-discrete norm defined in Theorem~\ref{thm oe}, our goal is to establish another stability estimate involving weighted-discrete norms:
\begin{equation}\label{eq stab ineq 2}
    C^{-1}h^{2q}\| u^h\|^2_{H^{q+2}(\Omega)}
    %\leq \| \mathcal{L}u^h\|^2_{L^2(\Omega)} + h^{-3} \| u^h\|^2_{L^2(\partial\Omega)}
      \leq % \stackrel{\eqref{eq stab ineq}}{\leq}
     \| \mathcal{L} u^h\|^2_{Y,\widetilde{W}_Y} +  h_{}^{-3} \| u^h\|^2_{Z,\widetilde{W}_Z}
       \text{\qquad for all $u^h \in U_{\Phi,X}$},
\end{equation}
which can be proven by overestimation of \eqref{eq stab ineq}.
This inequality holds for some zeroth-order weights $\widetilde{W}_Y,\widetilde{W}_Z$.
The solution to the associated kernel-based \emph{weighted least-squares}  collocation method is defined as
\begin{equation}\label{eq:WLS}
u^h_{{WLS}} \coloneqq \frac{1}{2} \arginf_{v\in U_{\Phi,X}}
\left( \| \mathcal{L} v {-f}\|^2_{Y,\widetilde{W}_Y} +  h_{}^{-3} \| v {-g}\|^2_{Z,\widetilde{W}_Z} \right),
\end{equation}
that requires solving {the} resultant overdetermined system \eqref{VLS mtx sys LS} in the least-squares sense.

As remarked in the previous section, Theorem~\ref{thm oe} has already demonstrated that the continuous and weighted-discrete norms on the boundary are equivalent in the trial space
\begin{equation}\label{bdy norm}
    \| u^h\|^2_{L^2(\partial\Omega)} \sim \| u^h\|^2_{Z,\widetilde{W}_Z}
     \text{\qquad for all $u^h \in U_{\Phi,X}$,}
\end{equation}
using constants $\frac14$ and $4$ in this norm equivalency.

Our remaining task is to handle the norms of $\mathcal{L}u^h\in H^{\tau-2}(\Omega)$ for any $u^h\in U_{\Phi,X}$. We need another kernel $\Psi$ that satisfies the decay condition \eqref{fourierdecay} with a rate of $\tau-2$ rather than $\tau$. For instance, if the Matérn kernel is used, we can simply set $\Psi = \Phi_{\tau-2}$ as defined in \eqref{eq:MaternKernel}.

%The problem is that, in general cases, applying the differential operator $\mathcal{L}$ to $u^h$ may result in a function that is not in any trial space.

{The problem is that, in general cases, applying the differential operator $\mathcal{L}$ to $u^h\in U_{\Phi,X}$ results in a function $\mathcal{L}u^h\not\in U_{\Phi,X}$ that is not in same trial space.}
To bound its $L^2$-norm using Theorem~\ref{thm oe}, we might consider  some projection or best approximation from $H^{\tau-2}(\Omega)$ to $U_{\Psi,X}$. However, this approach is flawed as the nodal values would change after projection,  leading to the disappearance  of the desired weighted-discrete norm in the upper bound from  the analysis.
%To circumvent this issue, our solution is to project using the sufficiently dense set of collocation points $Y$ and weight $\widetilde{{W}}_Y$ for Theorem~\ref{thm oe} to apply.
{
To circumvent this issue, our solution is to project using a sufficiently dense set of collocation points $Y$.
This denseness is quantitatively defined by the condition:
\begin{equation}\label{new denseness}
C(h_Y/ h_{})^{\tau-2 } < {1/8},
\end{equation}
and along with that for Theorem~\ref{thm oe} to apply and, hence, for the existence of the weight $\widetilde{{W}}_Y$ for the norm equivalence inequalities to hold.
}
Instead of $X$, we define the trial space $U_{\Psi,Y} \subset H^{\tau-2}(\Omega)$ on $Y$. We denote the interpolant of $v\in H^{\tau-2}(\Omega)$ on $Y$ in the trial space $U_{\Psi,Y}$ as $s_v=I_{\Psi,Y}v$.

For the sake of simplicity, we use $C$ to denote all generic constants that appear below, bearing in mind that these may assume different values in different contexts.
Using Theorem~\ref{thm oe} and an error estimate for kernel interpolation \cite{Narcowich+WardETAL-Sobobounfuncwith:05}, we get
\begin{eqnarray*}
    \| v\|^2_{L^2(\Omega)} &\leq & \| s_v \|^2_{L^2(\Omega)} + \| v-s_v\|^2_{L^2(\Omega)}
    \\
    &\leq& 4 \| s_v \|^2_{Y,\widetilde{W}_Y}  +C h_Y^{2\tau-4} \| v\|_{H^{\tau-2}(\Omega)}^2
    \\
    &\leq& 4 \| v \|^2_{Y,\widetilde{W}_Y}  + C h_Y^{2\tau-4} \| v\|_{H^{\tau-2}(\Omega)}^2
    \text{\qquad for all $v \in H^{\tau-2}(\Omega)$}.
\end{eqnarray*}
For any $u^h\in U_{\Phi,X}\subset H^{\tau}(\Omega)$, if we set $v= \mathcal{L}u^h$, we can derive the following inequality
\begin{eqnarray*}
  \| \mathcal{L}u^h\|^2_{L^2(\Omega)} &\leq & 4\| \mathcal{L}u^h \|^2_{Y,\widetilde{W}_Y}  + C h_Y^{2\tau-4} \| \mathcal{L}u^h\|_{H^{\tau-2}(\Omega)}^2
  \\
  &\leq &4\| \mathcal{L}u^h \|^2_{Y,\widetilde{W}_Y}  + C h_Y^{2\tau-4} \|u^h\|_{H^{\tau}(\Omega)}^2
  \\
  &\leq &4\| \mathcal{L}u^h \|^2_{Y,\widetilde{W}_Y}  + C h_Y^{2\tau-4}  h_{}^{-2\tau+4+2q}    \|u^h\|_{H^{2+q}(\Omega)}  ^2,
\end{eqnarray*}
which is the result of an inverse inequality found in \cite[Lemma~3.2]{Cheung+LingETAL-leaskerncollmeth:18} of our  prior work.
Now, we can then use this result to
bound \eqref{eq stab ineq} with generic constant $C_{\ref{eq stab ineq}}$
from the above as follows
\begin{eqnarray*}
  C_{\ref{eq stab ineq}}^{-1}h^{2q}\| u^h\|^2_{H^{q+2}(\Omega)}
  &\leq& 4 \Big( \| \mathcal{L}u^h \|^2_{Y,\widetilde{W}_Y} + h_{}^{-3} \| u^h\|^2_{Z,\widetilde{W}_Z} \Big) \\
  &&\qquad\qquad\qquad +  h_{}^{2q}  \Big(  C (h_Y/ h_{})^{\tau-2}  \|u^h\|_{H^{2+q}(\Omega)}\Big)  ^2,
\end{eqnarray*}
which holds for all $u^h\in U_{\Phi,X}$.
%Under a new denseness requirement by using yet another generic constant $C$:
Under {the denseness requirement \eqref{new denseness}} by using yet another generic constant $C$,
%\begin{equation}\label{new denseness}
%$$ C(h_Y/ h_{})^{\tau-2 }  < {1/8}$$
%\end{equation}
c.f., the one in \eqref{denseness} for the non-weighted version,  the weighted-discrete stability inequality in \eqref{eq stab ineq 2} holds for all $u^h\in U_{\Phi,X}$.
Since $\widetilde{{W}}$ only contains bounded quadrature weights, the $L^\infty$ consistency analysis from our previous work directly applies to the new collocation method \eqref{eq:WLS}. We refer readers to \cite{Cheung+LingETAL-leaskerncollmeth:18} for details.

Under the same denseness requirement \eqref{new denseness}, we can invert the above proof to
%\[
%        \| \mathcal{L}u^h\|^2_{L^2(\Omega)} \sim \| \mathcal{L}u^h \|^2_{Y,\widetilde{W}_Y}
%     \text{\qquad for all $u^h \in U_{\Phi,X}$.}
%\]
{
show the above bounds hold in reverse order:
\begin{eqnarray*}
  %\| \mathcal{L}u^h \|^2_{Y,\widetilde{W}_Y}
  \| v \|^2_{Y,\widetilde{W}_Y}
   &\leq&\| s_{v} \|^2_{Y,\widetilde{W}_Y} \leq  4 \| s_v \|^2_{L^2(\Omega)}
  \\ &\leq& 4 \Big( \| v\|^2_{L^2(\Omega)} + \|v- s_v \|^2_{L^2(\Omega)} \Big)
  \\ &\leq& 4 \Big( \| v\|^2_{L^2(\Omega)} +
  C  {h_Y^{2\tau-4} }\| v\|_{H^{\tau-2}(\Omega)}^2 \Big)
    \text{\qquad for all $v \in H^{\tau-2}(\Omega)$,}
\end{eqnarray*}
or, for any $u^h\in U_{\Phi,X}\subset H^{\tau}(\Omega)$, we have
\[
    \| \mathcal{L}u^h \|^2_{Y,\widetilde{W}_Y}
    \leq 4\| \mathcal{L}u^h \|^2_{L^2(\Omega)}
    + 4C h_Y^{2\tau-4}  h_{}^{-2\tau+4+2q}    \|u^h\|_{H^{2+q}(\Omega)}  ^2.
\]
}\noindent
When coupled with the boundary counterpart in \eqref{bdy norm}, we can use the zeroth-order quadrature weights $\widetilde{W}_Y$ and $\widetilde{W}_Z$, whose existence is ensured by Theorem~\ref{thm oe}, to establish  the norm equivalency as follows:
\begin{equation}\label{equiv stab}
    \| \mathcal{L}u^h\|^2_{L^2(\Omega)} + h^{-3} \| u^h\|^2_{L^2(\partial\Omega)}
    \sim
    \| \mathcal{L}u^h \|^2_{Y,\widetilde{W}_Y} + h^{-3} \| u^h\|^2_{Z,\widetilde{W}_Z}
     \text{\qquad for all $u^h \in U_{\Phi,X}$}.
\end{equation}
With this stronger observation, the convergence of  the  WLS  solutions in \eqref{eq:WLS} can be proven via that of the VLS  solutions \eqref{eq:ls var sol}. Consequently, Theorem~\ref{main thm} applies to \eqref{eq:WLS} {resulting in the following convergence theorem}.
%In the following section, we will demonstrate how certain straightforward strategies for approximating the base area in the Riemann sum can yield zeroth-order quadrature weights that show promising numerical results.
%{\begin{theorem}
%Under the assumptions of Theorem \ref{main thm} and results in Theorems \ref{thm stab} and \ref{thm oe},  the error estimate between all trial functions $u^h\in U_{\Phi,X}\subset H^\tau(\Omega)$ and the solution $u\in H^{q+2}(\Omega)$ to \eqref{eq:bvp} is
%\begin{equation}\label{equiv err}
%\|u-u^h\|^2_{H^{q+2}(\Omega)}
%\leq
%C\left(h^{2\tau-2q-4}+h^{2\tau-4}+h^{2\tau-3}\right)
%\| u \|^2_{H^{\tau}(\Omega)},  \quad 0\leq q \leq \tau-2,
%\end{equation}
% for with some constant $C$ that only depends on $q$, $X$, $\Omega$, $\Phi$ and $\mathcal{L}$.
%\end{theorem}}

In the following section, we will demonstrate how straightforward strategies for generating zeroth-order quadrature weights can yield promising numerical results.

%\begin{eqnarray*}
%  \| \mathcal{L}u^h \|^2_{Y,\widetilde{W}_Y}
%  &\leq &  \| s_{v} \|^2_{Y,\widetilde{W}_Y} \leq 4 \| s_v \|^2_{L^2(\Omega)}
%  \\ &\leq& 4 \Big( \| v\|^2_{L^2(\Omega)} + \| s_v -v\|^2_{L^2(\Omega)} \Big)
%  \\ &\leq& 4 \Big( \| v\|^2_{L^2(\Omega)} + \| s_v -v\|^2_{L^2(\Omega)} \Big)
%\end{eqnarray*}
% RBF+Inv Ineq gives the same denseness requirement

%write an opening paragraph for \section{Numerical Examples} with key points:
%0. it sounds crazy to not use any quadrature weight in variational formulation.
%in this section, we now provide arguements and numerical result to show that it is possible in the consider kernel based collocation setting.
%1. linear algebra manipulation can only handle one fix linear system \eqref{VLS mtx sys LS} at a time; thus, it cannot be used to deduce general convergence theory. Yet, provided that the condition number of ${W}$ is moderate, we can see that the WLS solution indeed solves a nearby VLS minimization problem and vice versa.
%2.  if any 0th order quad weights allow the equivalency in \eqref{equiv stab} to hold, then the LS variational solution \eqref{eq:ls var sol} and the weighted LS solution \eqref{eq:WLS} converges at the same rate, but subject to different generic constant. That is, the actual accuracy could be different.
\section{Numerical Examples}\label{sec num exmp}

While the idea of not using any quadrature weight in a VLS formulation might initially seem counterintuitive, this section aims to provide both theoretical arguments and numerical results demonstrating that such an approach is indeed feasible in our considered kernel-based collocation setting. This becomes particularly handy when considering the challenges associated with finding weights in higher dimensions and dealing with scattered data.

For the sets of data points $X$, $Y$, and $Z$ specified in each numerical example, let us detail all tested methods in terms of the linear system as shown in \eqref{VLS mtx sys LS}. All methods use a boundary scaling factor $\vartheta=h^{ {-\frac32}}$ and are defined as follows:
\begin{adjustwidth}{0.5cm}{}
\begin{description}
  \item[VLS-Tp:] Approximation to the variational least-squares solution \eqref{eq:ls var sol}, with integrals approximated by quadrature weights ${W}$ using the Trapezoidal rule.
  \item[WLS-Id:] Weighted least-squares solution \eqref{eq:WLS}, obtained by solving the system in \eqref{VLS mtx sys LS} with a constant weight ${W}=I$.
  \item[WLS-Rd:] This method is a counterpart to WLS-Id but uses a random weight ${W}$ where $\text{Diag}({W})$ is distributed uniformly in the range [0.5, 1.5].
\end{description}
\end{adjustwidth}
If the constant and random ${W}$, as zeroth-order quadrature weights, could  uphold the equivalency in \eqref{equiv stab}, then VLS-Tp and the two WLS solutions should exhibit the same convergence rate. However, they may be dictated by different generic constants. As a result, despite similar convergence rates, the actual accuracy of these methods may vary.

Denote  $\alpha_{\mathcal{W}}$ and $\alpha_I$ {as} the coefficient vectors of the WLS solutions with %some
{any}
positive weights $\mathcal{W}$ and $I$, respectively.  Let us also rewrite the linear system \eqref{VLS mtx sys LS} in a simplified form  $A\alpha=b$. We then have the following:
\begin{eqnarray*}
  \|{\mathcal{W}}(A\alpha_\mathcal{W}-b)\| &\leq&  \|{\mathcal{W}}\|\,\|A\alpha_\mathcal{W}- b\|
  \\ &\leq& \|{\mathcal{W}}\|\,\|A\alpha_I- b\| = w_{\max}  \|A\alpha_I- b\|
.
\end{eqnarray*}
By reversing the argument, we find that $\|A\alpha_I-b\| \leq  {w_{\min}^{-1}} \|{\mathcal{W}}(A\alpha_\mathcal{W}- b)\|$.
If we substitute $\alpha_I$ into the weighted objective function, we get
\[
    \|{\mathcal{W}}(A\alpha_I-b)\| \leq w_{\max}\|A\alpha_I- b\| \leq \kappa(\mathcal{W})\|A\alpha_I- b\|.
\]
This suggests that if {the condition number} $\kappa(\mathcal{W})$ is small, then the WLS-Id solution almost solves the problem with weight $\mathcal{W}$. Conversely, we also have
\[
    \|A\alpha_\mathcal{W}- b\| \leq {w_{\min} }  \|{\mathcal{W}}(A\alpha_\mathcal{W}- b)\|  \leq \kappa(\mathcal{W})  \|{\mathcal{W}}(A\alpha_\mathcal{W}- b)\|.
\]
It is important to note that linear algebra manipulations can only address a single linear system at a time. Therefore, they cannot be used to derive a general convergence theory. But, as long as $\kappa(\mathcal{W})$ is moderate, the above calculations  show that the WLS-Id solution nearly solves the WLS optimization problem with weights $\mathcal{W}$, achieving close (but of course, somewhat larger) objective values, and vice versa. This observation suggests that WLS-Id and WLS-Rd always exhibit similar numerical performance, a claim we intend to verify. Additionally, this motivates us to conduct tests on collocation points $Y$ and $Z$ with large mesh ratios, which lead to an increase in the condition numbers of the Trapezoidal weights as $N_Y$ increases.

\subsection*{Example 1 (Observing predicted convergence rates)}
For the sake of reproducibility and to provide clear evidence supporting our theories, we focus on simple numerical setups. We consider the domain $\Omega=(-1,1)^2$ with regular trial centers $X$, varying the squared number $N_X$ up to $81^2$. Interior collocation points $Y$ are constructed by tensor product grid data, allowing for straightforward calculations to obtain the Trapezoidal weights. Here is the exact procedure:
For each oversampling ratio $\gamma \in {1, 1.5, \ldots, 4}$, we generate an initial regular grid $P\subset \overline{\Omega}$ with $N_Y=\lceil({\gamma N_X})^{1/2}\rceil^2 \approx \gamma N_X$. Then, we apply a coordinate  transform $T:\overline{\Omega}\to \overline{\Omega}$ to $P$, and set $Y=T(P)$ and $Z=T(P)\cap \partial\Omega$.
The fill distances $h_Y$ and $h_Z$ for these transformed points are of the same magnitude. In this example, we do not consider varying the $h_Y/h_Z$ ratio based on the numerical results from our prior work \cite{Cheung+LingETAL-leaskerncollmeth:18}. In there, the numerical method is WLS with ${W}=I$ and boundary scaling factors
$\vartheta= \big( (h_Z/h_Y)^{(d/2-q)\theta} h_Z^{-2\theta} \big)^{1/2}$
for all $\theta\geq0$. Aside from the problem of ill-conditioning, the WLS methods with a full range of scaling factors demonstrated the same $H^2$-convergence rate. The effects of using different scaling factors in that context are similar to the effects of varying the $h_Y/h_Z$ ratio in this study.
Two test problems were considered:\smallskip
\begin{description}
\item[PDE 1.] We consider the modified Helmholtz equation $\Delta u - u = f$ subject to the Dirichlet boundary condition. The exact solution $u^*  = \text{sech}(\pi(x-0.5)) + \text{sech}(\pi(y-0.75))$ is used to generate PDE data. The sine-transform \cite{Tang+Trummer-Bounlayeresopseu:96}, defined component-wise by $T(P)=\sin( \pi P/2)$, is used in this case.

\item[PDE 2.] We examine the Poisson equation $\nabla \boldsymbol\cdot( c\, \nabla u ) = f$ with a strictly negative diffusion tensor $c(x,y)=-\arctan(100(x+y))-3\pi/4$, subject to the Dirichlet boundary condition. The exact solution is taken to be $u^* =(1-x^2)\cos(\pi y/2)$. Here, we use the componentwise signed-square transform  $T(P)=\text{sgn}(P)P^2$ to define collocation points and collocate the PDE in non-divergence form.
\end{description}
Note that the transformations $T$ selected in PDEs 1 and  2 were chosen for demonstration purposes, and not based on specific information from the PDEs themselves. Figure~\ref{fig Y}(a)--(b) shows the distributions of $N_Y=40^2$ sine-transformed and signed-squared collocation points $Y$ for PDEs 1 and 2 respectively. The background color in both subfigures represents the color contour of the PDE source function $f$.

\begin{figure}%[!p]
	\centering
	\begin{tabular}{ccc}
		\begin{overpic}
			[width=.31\textwidth,trim= 50 3 70 5, clip=true,tics=10 ]{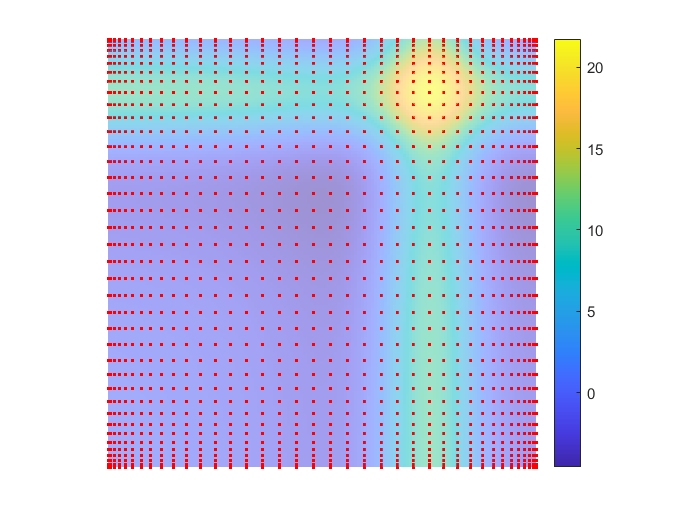}
			\put(20,2) {\scriptsize (a) $Y$ for PDE 1}
		\end{overpic}
		&
		\begin{overpic}
			[width=.31\textwidth,trim= 50 3 70 5, clip=true,tics=10 ]{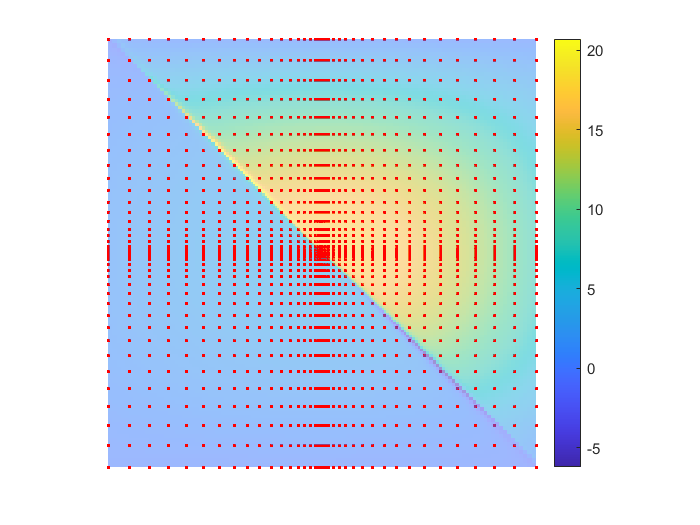}
			\put(20,2) {\scriptsize (b) $Y$ for PDE 2}
		\end{overpic}
		&
		\begin{overpic}
			[width=.228\textwidth,trim= 85 3 80 5, clip=true,tics=10 ]{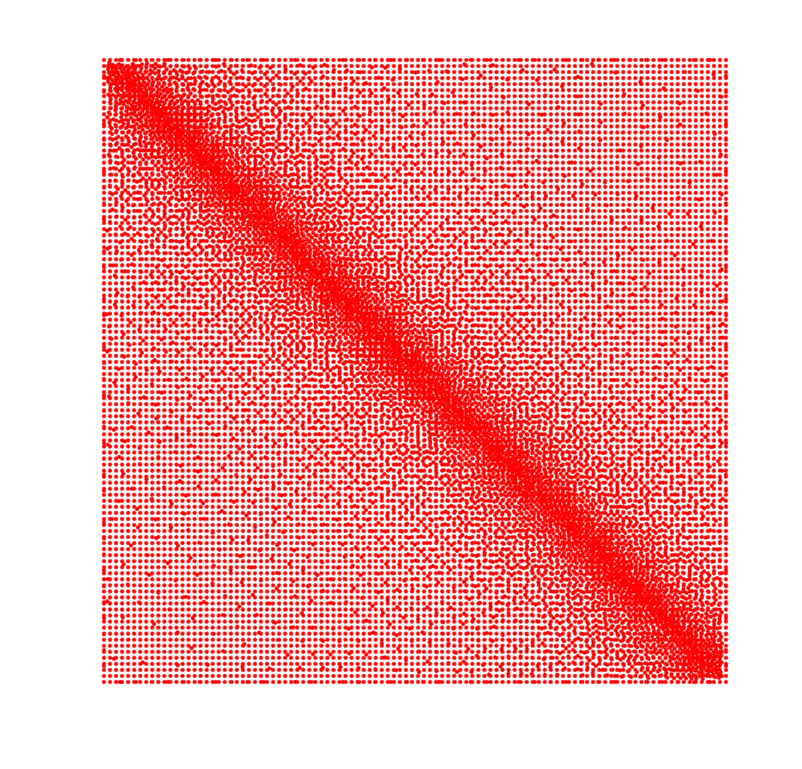}
			\put(10,2) {\scriptsize (c) $Y_2$ for PDE 4}
		\end{overpic}
	\end{tabular}
	\caption{Distribution of $N_Y=40^2$ sine-transformed (PDE 1), signed-squared (PDE 2), and NodeLab \cite{Mishra-Node:19} generated (PDE4) collocation points, overlaying the color contour of the corresponding source function $f$ (PDEs 1 and 2).}\label{fig Y}
\end{figure}

We use the Mat\'{e}rn kernel  \eqref{eq:MaternKernel} with orders $\tau=4,5,6$ and scale $r \leftarrow  \epsilon r$ with a shape parameter $\epsilon=5$ for all $\tau$. Such selections of smoothness order and shape parameters allow us to observe typical behavior in kernel-based collocation methods.
We report relative $L^2$ errors, which are computed using $\ell^2(P)$-norms of the nodal values of the error (or difference) function and the exact (or reference) solution on a regular point set $P$ consisting of $86^2$ points.
The convergence profiles for PDE 1 and PDE 2 are shown in Figures~\ref{fig err1} and \ref{fig err2}, respectively, each containing three subfigures for each kernel smoothness $\tau$.
Errors associated with a small oversampling ratio $\gamma \leq 2$ are shown individually, while those with $\gamma\geq 2.5$ are collectively represented in shaded areas.
The predicted convergence rate from Theorem~\ref{main thm} is simply $\tau$, and a reference slope has been added for ease of comparison.

\begin{figure}%[!p]
	\centering
		\begin{overpic}
			[height=.65\textwidth,trim= 70 85 80 50, clip=true,tics=10]{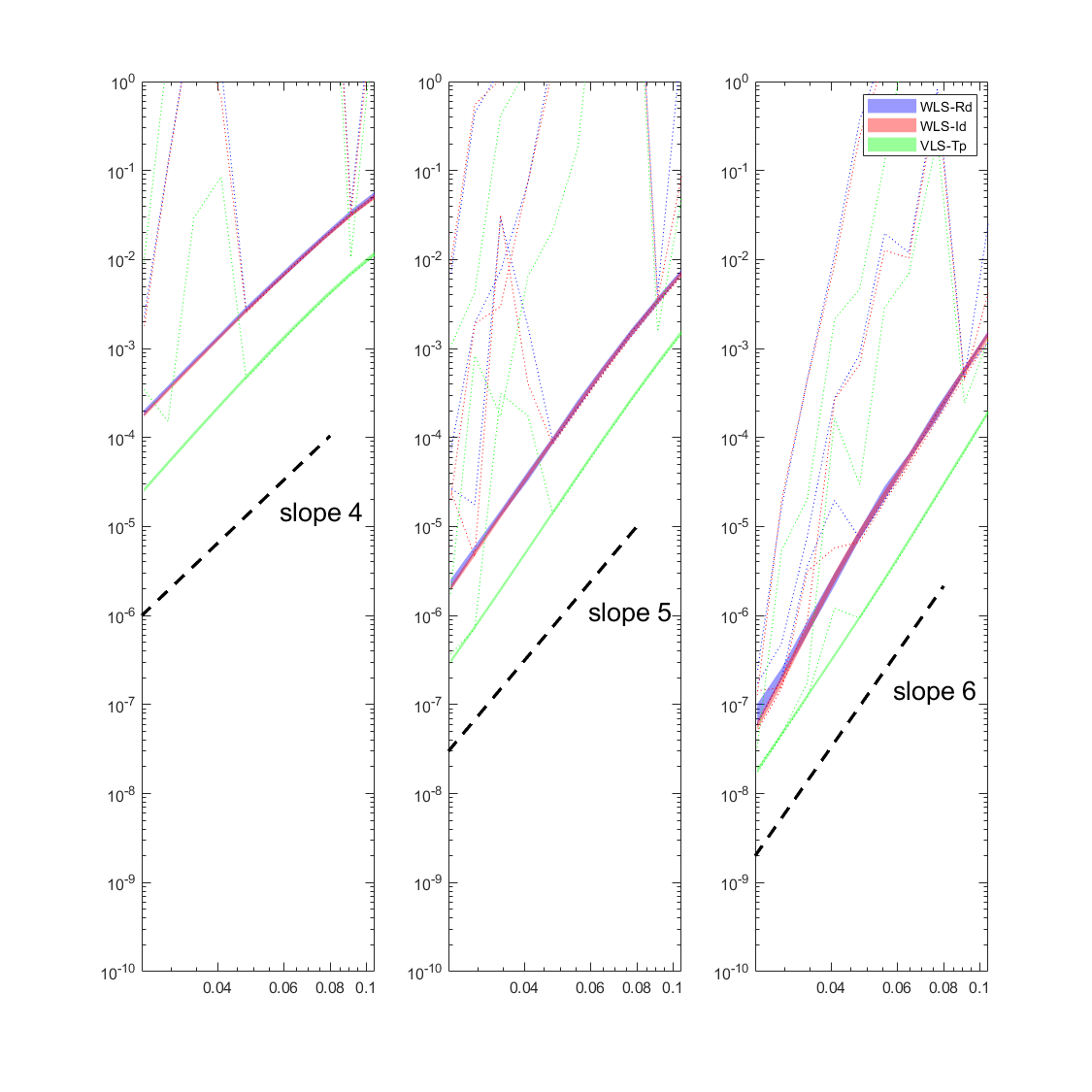}
			\put(-5,40) {\rotatebox{90}{PDE 1}}
            \put(14.5,-2) {\scriptsize$\tau=4$}
            \put(48.0,-2) {\scriptsize$\tau=5$}
            \put(81.5,-2) {\scriptsize$\tau=6$}
		\end{overpic}
	\caption{Relative $L^2$ convergence profiles of VLS-Tp, WLS-Id, and WLS-Rd for PDE 1 using the Mat\'ern kernel with smoothness orders $\tau=4,5,6$.
Errors associated with oversampling ratio $\gamma \in \{1.0,1.5,2.0\}$ are shown individually, while those with $\gamma\in\{ 2.5,3.0,3.5,4.0\}$ are collectively represented in a shaded area.
The predicted convergence rate from Theorem~\ref{main thm} is represented by the reference slope.}\label{fig err1}
\end{figure}
\begin{figure}%[!p]
	\centering
		\begin{overpic}
			[height=.65\textwidth,trim= 70 85 80 50, clip=true,tics=10]{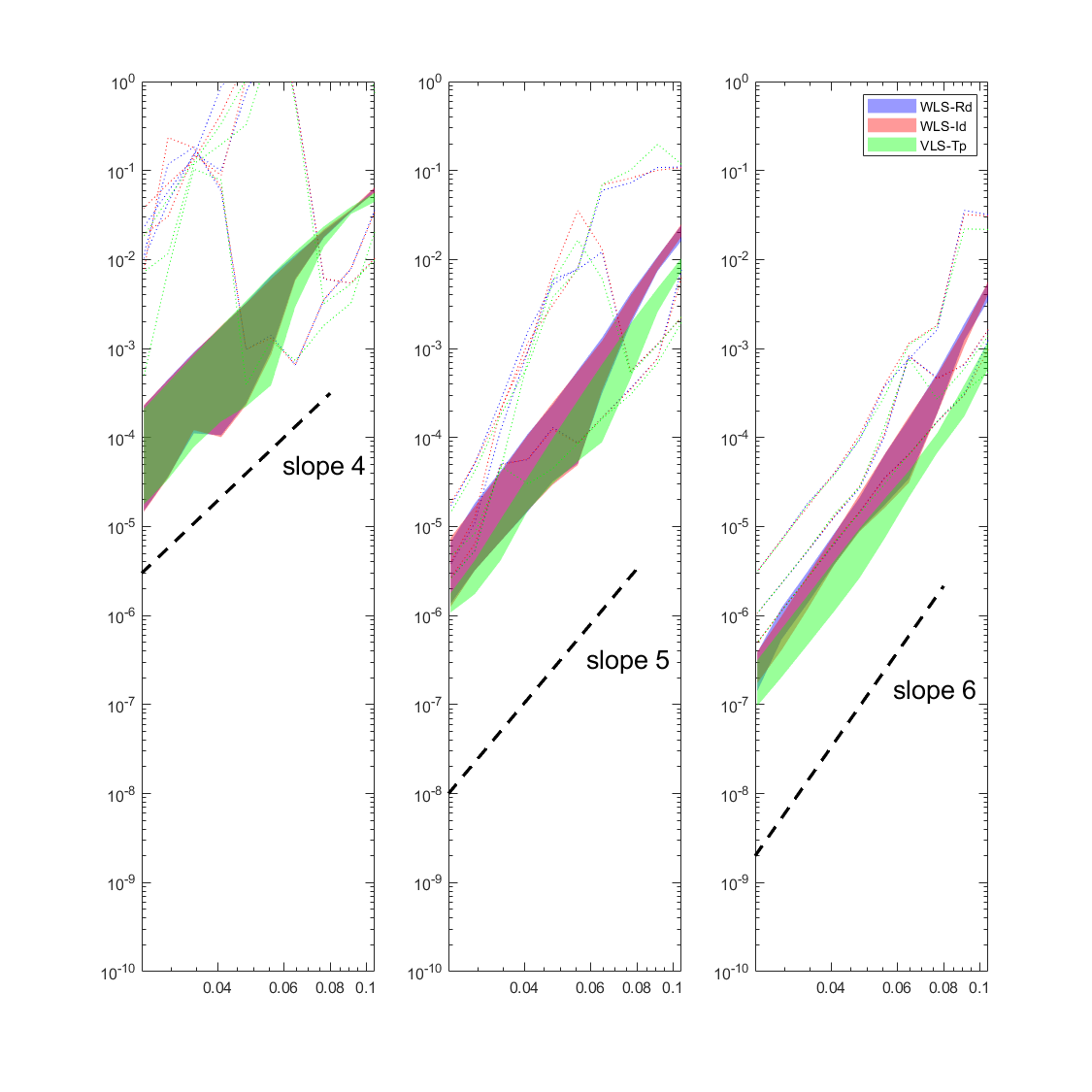}
			\put(-5,40) {\rotatebox{90}{PDE 2}}
            \put(14.5,-2) {\scriptsize$\tau=4$}
            \put(48.0,-2) {\scriptsize$\tau=5$}
            \put(81.5,-2) {\scriptsize$\tau=6$}
		\end{overpic}
	\caption{Relative $L^2$ convergence profiles for PDE 2, presented in the same format as in Figure~\ref{fig err1}.}\label{fig err2}
\end{figure}
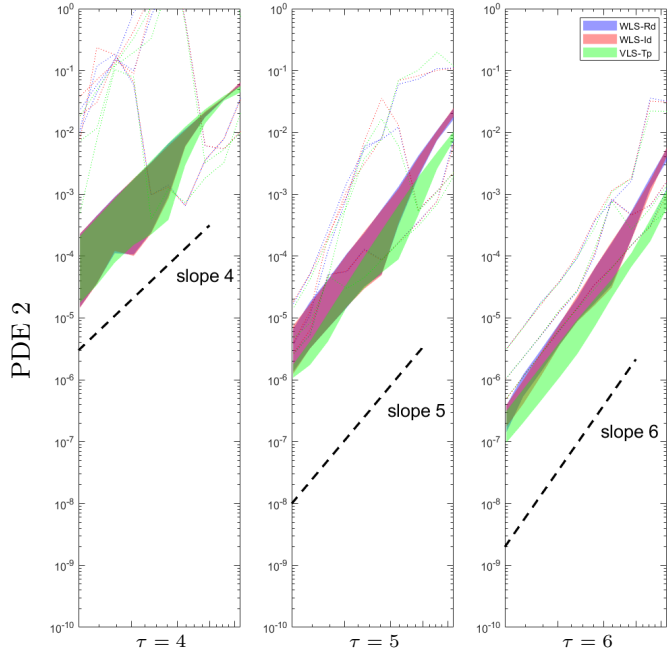

PDE 1 is a relatively easy test problem owing to its smooth and artificial analytic solution. From Figure~\ref{fig err1}, it is clear that a sufficiently high oversampling ratio is necessary. Most importantly, the two WLS-xx methods and VLS-Tp all converge at the same rate, a result that validates the convergence theory of WLS as discussed in Section~\ref{sec WLSconv}. Notably, the convergence is not impacted by the increasing condition number $\kappa({W})$ of the Trapezoidal weight ${W}$, which grows at a rate of $\mathcal{O}(N_Y^3)$ in both coordinate transforms. The accuracy of WLS-Id and WLS-Rd is nearly identical.

Conversely, PDE 2 is more difficult, even with its
%agreeable
analytic solution. This is attributed to the source function that changes rapidly along the line $x+y=0$. The error profiles are not as linear as in PDE 1, yet the trend aligns with the forecasted convergence rate. The accuracy advantage of VLS-Tp over the other two WLS-xx methods is not immediately apparent, but it appears to be more accurate when $\tau=6$.

\subsection*{Example 2 (Superconvergence with zeroth order quadrature)}
Superconvergence in the context of kernel-based interpolation \cite{Schaback-Supekerninte:18} is a complex interplay between the smoothness and the localization of the function. This example aims  to replicate similar results within the framework of VLS and WLS methods. \smallskip%%
\begin{description}
\item[PDE 3.] We consider the modified Helmholtz equation $\Delta u - u = f$ subject to the Dirichlet boundary condition. The exact solution is a  Gaussian function $u^*(x,y)=\exp(-10[(x-x_0)^2+(y-y_0)^2])$. The signed-square transform is used to increase collocation point density near the origin.
\end{description}
In the first test, we set $[x_0,y_0]=[0,0]$ which results in a high density of collocation points within the support of the Gaussian. In the second test, we choose $[x_0,y_0]=[0.5,0.75]$ for a contrary setup with lower point density. The resulting error profiles for both scenarios are displayed in Figures~\ref{fig err3ctr} and \ref{fig err3sft}.
Figure~\ref{fig err3ctr} exhibits high accuracy and clear superconvergence, demonstrating the benefits of having a high-density collocation point set at the right place. On the other hand, Figure~\ref{fig err3sft} shows an even faster convergence rate, but the approximation accuracy compared to Figure~\ref{fig err3ctr} is less. This highlights the influence of point density on the balance between convergence rate and approximation accuracy in the VLS and WLS methods. All findings here are consistent with our observations in Example~1, where all methods converged at the same rate.

\begin{figure}%[!p]
	\centering
		\begin{overpic}
			[height=.62\textwidth,trim= 70 85 80 50, clip=true,tics=10]{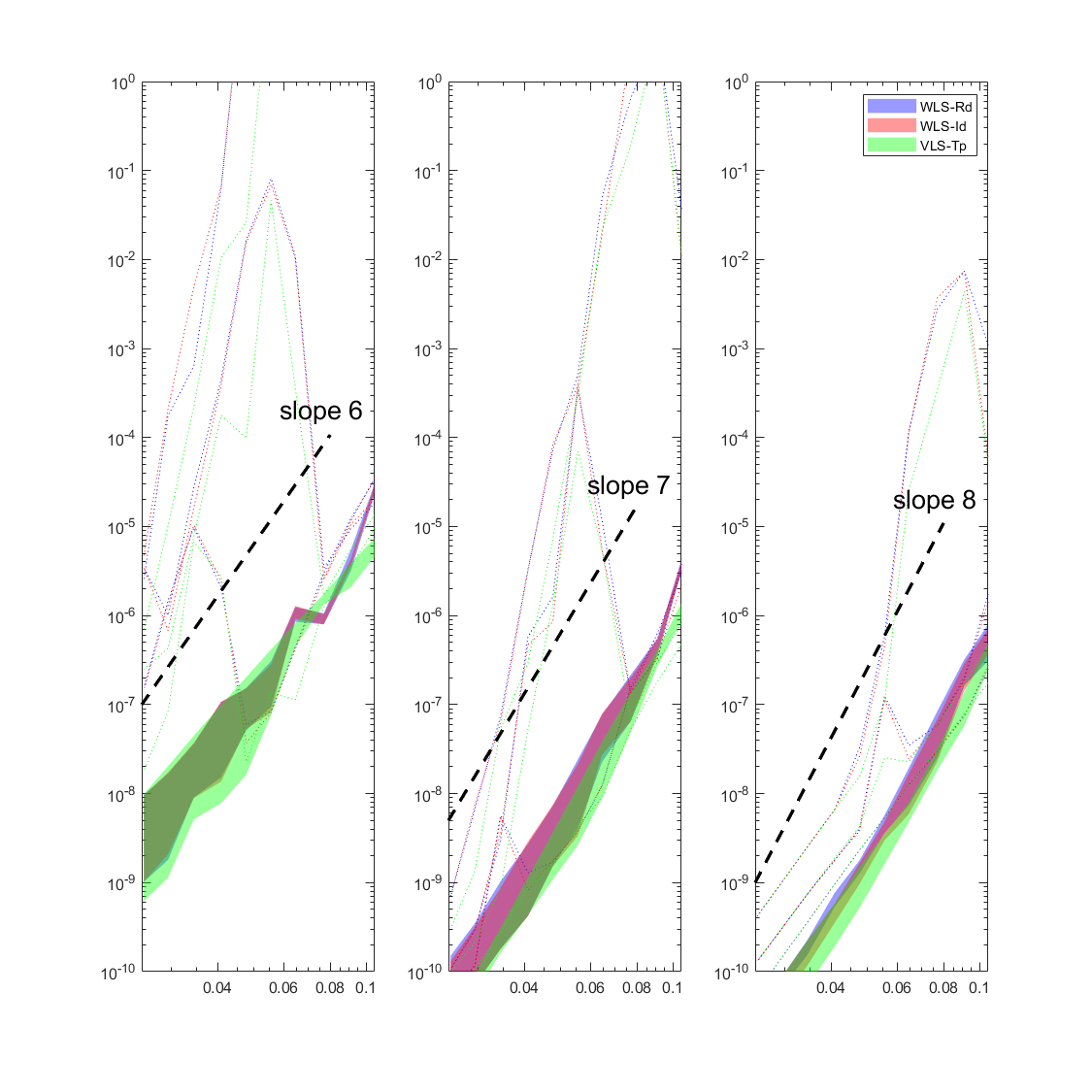}
			\put(-5,30) {\rotatebox{90}{PDE 3 with $[x_0,y_0]=[0,0]$}}
            \put(14.5,-2) {\scriptsize$\tau=4$}
            \put(48.0,-2) {\scriptsize$\tau=5$}
            \put(81.5,-2) {\scriptsize$\tau=6$}
		\end{overpic}
	\caption{Relative $L^2$ convergence profiles for Example 2, solving the modified Helmholtz equation $\Delta u - u = f$ with a Gaussian function centered at $[0,0]$ as the exact solution, presented in the same format as in Figure~\ref{fig err1}.}\label{fig err3ctr}
%\end{figure}
%\begin{figure}%[!p]
	\centering
		\begin{overpic}
			[height=.62\textwidth,trim= 70 85 80 50, clip=true,tics=10]{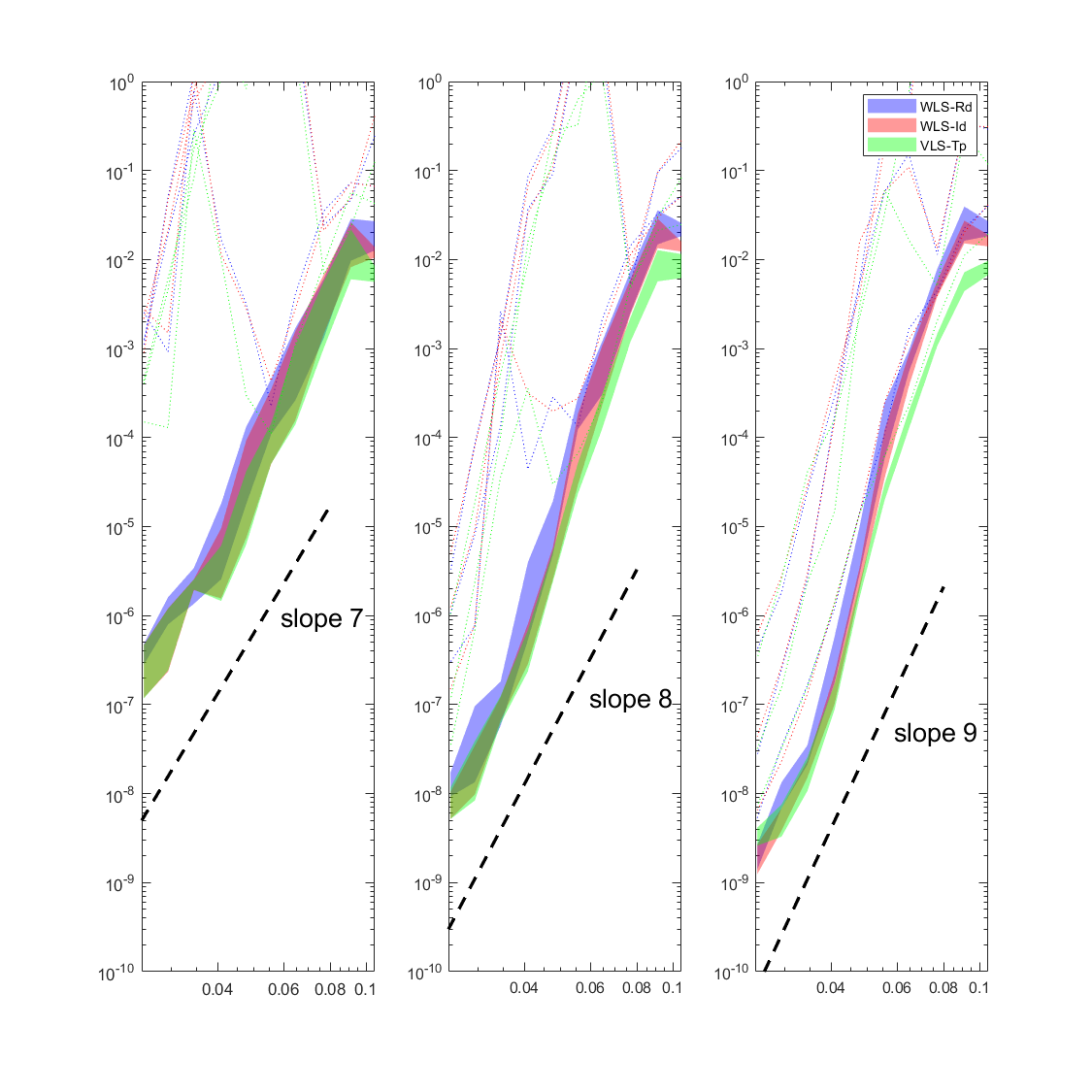}
			\put(-5,30) {\rotatebox{90}{PDE 3 with $[x_0,y_0]=[0.5,0.75]$}}
            \put(14.5,-2) {\scriptsize$\tau=4$}
            \put(48.0,-2) {\scriptsize$\tau=5$}
            \put(81.5,-2) {\scriptsize$\tau=6$}
		\end{overpic}
	\caption{Relative $L^2$ convergence profiles for the same equation as in Figure~\ref{fig err3ctr} but with the Gaussian function centered at $[0.5,0.75]$.}\label{fig err3sft}
\end{figure}

\subsection*{Example 3 (Collocation in divergence form with nearly singular diffusion tensor)}
In this final example, we examine a PDE characterized by a simple data function, but without an exact solution:\smallskip
\begin{description}
\item[PDE 4.] We consider the Poisson equation $\nabla \boldsymbol\cdot( c \,\nabla u ) = 1$, equipped with a strictly negative diffusion tensor $c(x,y)=-\arctan(100(x+y))-3\pi/4$, and subject to the Dirichlet boundary condition $u_{|\partial\Omega}=0$. Given a set of trial center $X$, we employ an oversampling ratio $\gamma=2.5$ to generate regular $Y_1$ and $Z$. The NodeLab \cite{Mishra-Node:19} package is used to produce scattered 7338 data points as $Y_2$, densely populating along $x-y=0$, as illustrated in Figure~\ref{fig Y}(c). Then, we take $Y=Y_1\cup Y_2$ as the set of PDE collocation points. The PDE is collocated in divergence form.
\end{description}
In these examples, we only focus on divergence form PDEs because non-divergence form PDEs present considerable challenges for many methods, including both our VLS-Tp (on some tensor grid $Y$) and WLS-Id methods and also RBF-FD methods. In fact, these methods often fail to solve non-divergence form PDEs correctly.

To validate the quality of our numerical approximation,  we compute a reference solution using the finite element method with highly refined discretizations of 19968 triangles and 10145 nodes.
Figure~\ref{fig sol4} shows the WLS-Id solutions for PDE 4 with a kernel smoothness order of $\tau=5$, calculated for $N_X=41^2$ and $81^2$ with surface color based on the absolute difference between the WLS-Id solution and  the reference solution.
We use a single color range determined by $N_X=41^2$ in both subfigures for easy  comparison.
In these runs, $(N_Y,N_{Y_1},N_Z)$ are $(18154,\,10816,\,412)$ and $(48954,\,41515,\,812)$ for $N_X=41^2$ and $N_X=81^2$ respectively. Even though there are relatively fewer collocation points on the boundary (which can often lead to lower accuracy), this fact does not cause significant accuracy problems.

\begin{figure}%[!p]
	\centering
		\begin{overpic}
			[height=.45\textwidth,trim= 30 50 40 20, clip=true,tics=10]{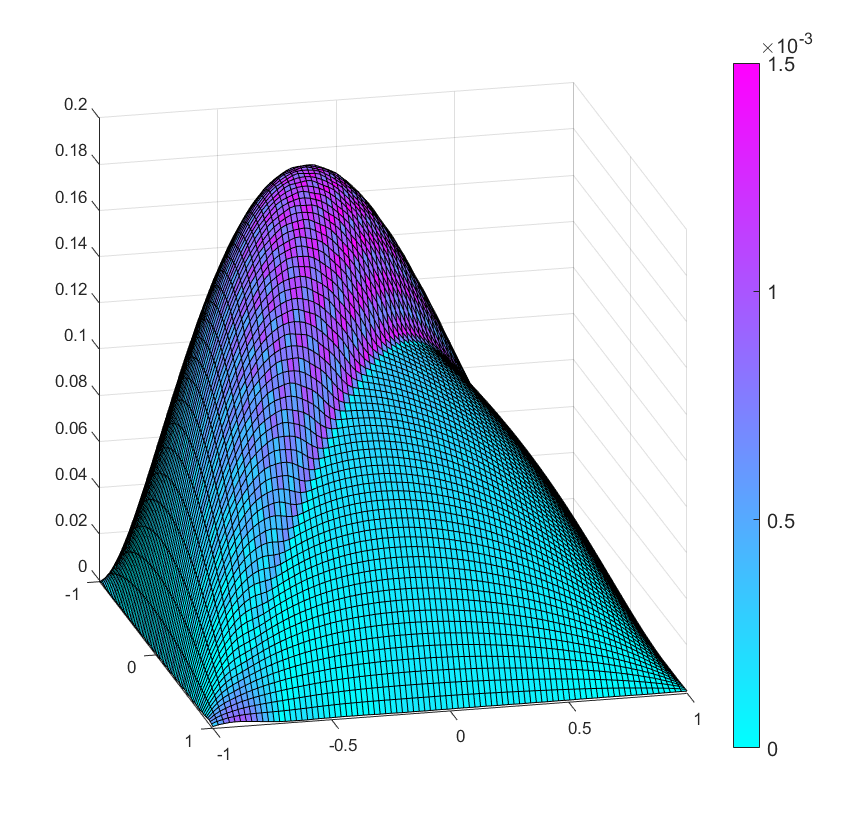}
            \put(30,-3) {\scriptsize (a) $N_X=41^2$}
		\end{overpic}
		\begin{overpic}
			[height=.45\textwidth,trim= 30 50 100 20, clip=true,tics=10]{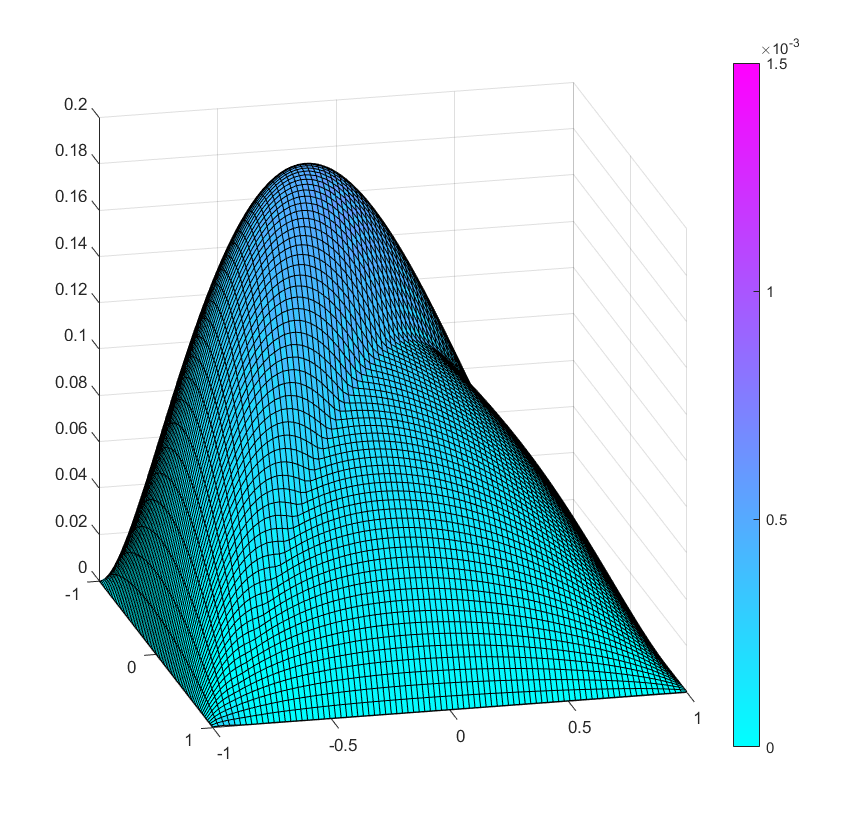}
            \put(30,-3) {\scriptsize (b) $N_X=81^2$}
		\end{overpic}
	\caption{WLS-Id solutions for PDE 4 with kernel smoothness order $\tau=5$ for (a) $N_X=41^2$ and (b) $81^2$. The color map represents the absolute difference between the computed WLS-Id solution and the reference solution obtained with a highly refined finite element method.}\label{fig sol4}
\end{figure}

\section{Conclusion}\label{sec conc}

In this paper, we have presented a comprehensive theoretical analysis of variational and weighted least-squares kernel-based methods, filling a significant gap in the existing literature. We provided rigorous proofs for two stability inequalities, crucial for understanding the convergence behavior of these methods.
Our work demonstrated that exact quadrature weights play a non-essential role in convergence and  the accuracy of the solutions is not sensitive to the ratio of interior to boundary collocation points.
These theoretical findings were supported  by a series of numerical experiments, which showcased the efficiency and accuracy of these methods on data sets with sufficiently large mesh ratios. The results confirmed our theoretical predictions regarding the performance of the variational method, the kernel-based collocation method, and our novel weighted least-squares collocation method. In all cases, we observed similar performance in variational and weighted least-squares kernel-based methods when the oversampling ratio was large.

\section*{Acknowledgement}

This work was supported by
the General Research Fund (GRF No. 12301419, 12301520, 12301021) of Hong Kong Research Grant Council,
the Opening Project of Guangdong Province Key Laboratory of Computation Science at the Sun Yat-sen University (Project No. 2021014), National Natural Science Foundation of China (NSFC No. 12361086, 12001261, 12371379), Jiangxi Provincial Natural Science Foundation (No. 20212BAB211020)
and Changsha Natural Science Foundation (Project No. KQ2202069).

%\section*{References}
\bibliographystyle{elsarticle-num}
\bibliography{CLS-Stability_R2-minimal}%,lling,others}

\end{document}